\documentclass[12pt,bezier]{article}
\usepackage{amssymb}
\usepackage{mathrsfs}
\usepackage{amsmath}
\usepackage{amsfonts,amsthm,amssymb}
\usepackage{amsfonts}
\usepackage{graphics}
\usepackage{graphicx}
\usepackage{color}
\usepackage{ifpdf}
\usepackage{fancyhdr}
\usepackage{epstopdf}
\usepackage{epsfig}
\usepackage{indentfirst}
\usepackage{CJK}

\newtheorem{lem}{Lemma}[section]
\newtheorem{thm}[lem]{Theorem}

\begin{document}

	\title{ A characterization of 4-connected graphs with no $K_{3,3}+v$-minor}
	\author{ Linsong  Wei, \quad Yuqi  Xu, \quad Weihua Yang\footnote{Corresponding author. E-mail: ywh222@163.com,~yangweihua@tyut.edu.cn}, \quad Yunxia Zhang\\
		\\ \small Department of Mathematics, Taiyuan University of
		Technology,\\
		\small  Taiyuan Shanxi-030024,
		China\\}
	\date{}

	\maketitle
\noindent{\bf Abstract:}
		Among graphs with 13 edges, there are exactly three internally 4-connected graphs which are $Oct^{+}$, cube+e and $ K_{3,3} +v$. A complete characterization of all 4-connected graphs with no $Oct^{+}$-minor is given in [John Maharry, An excluded minor theorem for the octahedron plus an edge, Journal of Graph Theory 57(2) (2008) 124-130]. Let $K_{3,3}+v$ denote the graph obtained by adding a new vertex $v$ to $K_{3,3}$ and joining $v$ to the four vertices of a 4-cycle. In this paper, we determine all 4-connected graphs that do not contain $K_{3,3}+v$ as a minor.
		
	\noindent{\bf Keywords:}  forbidden minor,  4-connected graph,  $ K_{3,3} +v$ 
	\section{Introduction}
		All graphs in this paper are simple. Let $G$ and $H$ be two graphs. $H$ is called a $ minor $ of $G$ if it can be generated by deleting or contracting edges from $G$. $G$ is called $H$-$ minor $-$ free $ if no minor of $ G $ is isomorphic to $H$. The Robertson-Seymour Graph Minors project has shown that minor-closed classes of graphs can be described by finitely many forbidden minors. We can get some graph classes which have many interesting properties in the process of excluding small minors. In addition, many important problems in graph theory are about the property of the $H$-minor-free graphs. For instance, Tutte$^{'}$s 4-flow conjecture asserts that every bridgeless Petersen-minor-free graph admits 4-fiow.
		
		Ding \cite{smallgraph} surveyed all $H$-minor-free graphs for 3-connected $H$ with at most 11 edges, including $Oct$$\backslash$$e$. For 3-connected graphs with 12 edges, Maharry \cite{cube} characterized 3-connected cube-minor-free graphs, Maharry and Ding \cite{oct,oct1} characterized 3-connected $Oct$-minor-free graphs, and Maharry and Robertson \cite{wanger} characterized the $V_8$-minor-free graphs. There are 51 3-connected graphs with 13 edges. Among the graphs with 13 edges, there are exactly three internally 4-connected graphs. These three graphs are $Oct^{+}$, cube+e and  $K_{3,3} +v$. The $Oct^{+}$ is a 4-connected graph which is obtained from the octahedron by adding an edge. Maharry characterized all 4-connected $Oct^{+}$-minor-free graphs in \cite{octe}. The cube+e is obtained by adding a long diagonal to the cube. In this paper, we characterize the 4-connected $K_{3,3}+v$-minor-free graphs, where $K_{3,3}+v$ is obtained by adding a new vertex $v$ to $K_{3,3}$ and joining $v$ to the four vertices of a 4-cycle in this paper.
		
		We use $G/e$ to denote the graph obtained from $G$ by first contracting $e$ and then deleting all but one edge from each parallel family and $G$$\backslash$$e$ to denote the graph obained from $G$ by deleting the edge $e$. For each integer $n\ge3$, let $DW_n$ denote a {\it double-wheel}, which is a graph on $n+2$ vertices obtained from a cycle $C_n$ by adding two adjacent vertices and connecting them to all vertices on the cycle. Let $\mathcal {DW}=\{DW_n:n\ge3\}$. For each integer $n\ge5$, let $C_n^2$ be a graph obtained from a cycle $C_n$ by joining all pairs of vertices of distance two on the cycle. Let $\mathcal C_0=\{C_{2n}^2:n\ge3\}$, $\mathcal C_1=\{C_{2n+1}^2:n\ge2\}$, and $\mathcal C=\mathcal C_0\cup\mathcal C_1$. The graph $L(H)$ is called the $line $ $graph$ of $G$ if $V(L(G))=E(G)$, and for any two vertices $e$, $f$ in $V(L(G))$, $e$ and $f$ are adjacent vertices if and not only if they are adjacent edges in $G$.
		
			\begin{thm}\label{thm1.2}
		A 4-connected graph $G$ is $ K_{3,3}+v $-$ minor $-$ free $ if and only if either G is planar or G belongs to $\mathcal {DW}$ $\cup$ $\mathcal C_1$ $\cup$ $\mathcal {M}$ $\cup$ $\{$$L(K_{3,3})$, $K_6$, $K_6$$\backslash$$e$, $DW_4$, $\Gamma_1$, $\Gamma_2$, $\Gamma_3$$\}$, where $\Gamma_1$,$\ldots$,$\Gamma_3$ are the three graphs shown below.
		\end{thm}
	
\begin{figure}[!htb]
	\centering
	\includegraphics[width=0.7\linewidth]{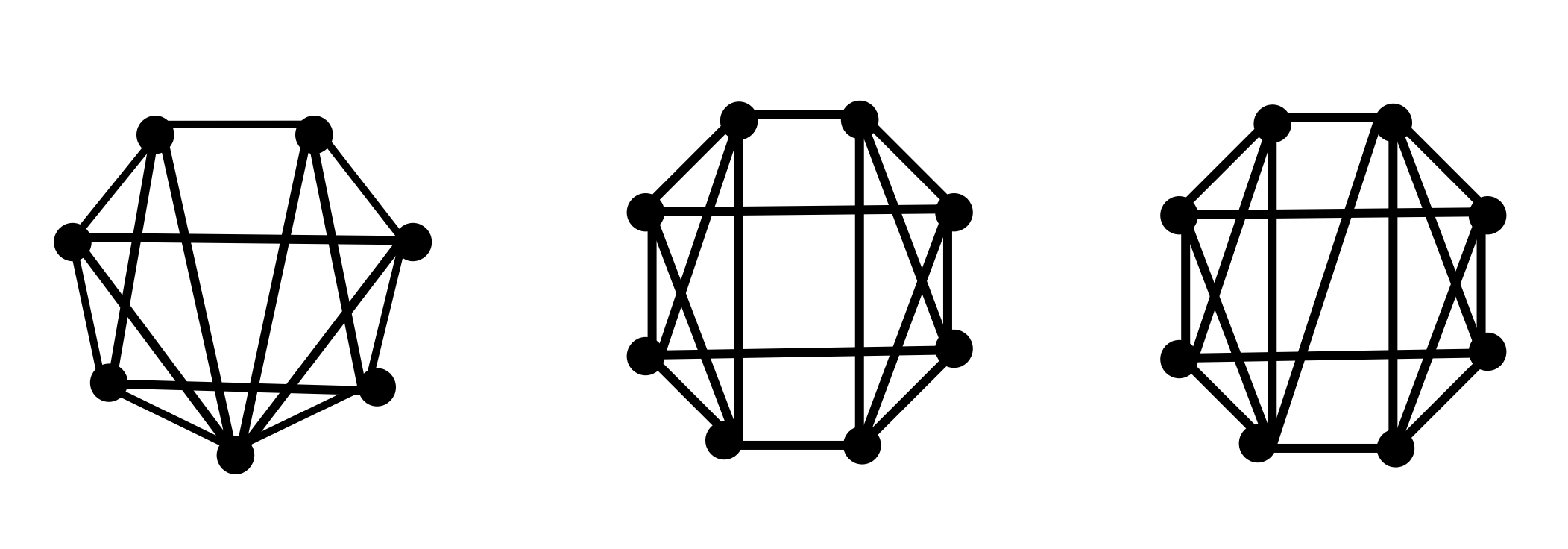}
	\caption{ $\Gamma_1$, $\Gamma_2$, $\Gamma_3$}
	\label{fig:01}
\end{figure}

	\section{Preliminaries}
       Let $v$ be a vertex of a 4-connected graph $G$. Let $N_G(v)$ denote the set of vertices of $G$ that are adjacent to a vertex $v$. A $4$-$split$ of $v$ produces a new $G^{'}$ as follows. Given two sets, $A, B$$\subseteq$$N_G(v)$, where $A$ $\cup$ $B$=$N_G(v)$ and min\{$|A|,|B|$\}$\geq$ 3, the graph $G^{'}$ is obtained by replacing the vertex $v$ in $G$ with vertices $a$ and $b$ such that $N_{G^{'}}(a)$=$A$ $\cup$ \{$b$\} and  $N_{G^{'}}(b)$=$B$ $\cup$ \{$a$\}. Note that the graph $G^{'}$ is also 4-connected.
       
       The following theorem of Martinov is a classical result to charactize all 4-connected graphs.
       
       \begin{thm}[\cite{4c}] \label{thm:chain1}
       Let G be a 4-connected graph. There exists a sequence of 4-connected graphs $H_0,\ldots,H_t$ $($t$\ge$$0)$ such that G is isomorphic to $H_0$ and $H_i$ is obtained from $H_{i-1}$ by contracting some edges e $\in$ E $($$H_{i-1}$$)$ and deleting any resulting parallel edges. Moreover, $H_t$ is either isomorphic to $C_n^2$ for some $n \ge 5$ or isomorphic to the line graph of a cubic cyclically 4-connected graph.
       \end{thm}
     
     This result has been strengthened by Qin and Ding as follows , which is an important tool to characterize all 4-connected graphs without $ K_{3,3}+v $ as a minor.
     
     \begin{thm}[\cite{chain}] \label{thm:chain2}
     Let G be a 4-connected graph not in  $\cal C\cup L$. If G is planar,then there exists  a $($G,$C_{6}^{2}$$)$-chain; if G is non-planar, then there exists a a $($G,$K_5$$)$-chain.
     \end{thm}
 
 We obtain a partial work in \cite{jia} for the characterization of 4-connected $K_{3,3}+v$-minor-free graphs.
   
    \begin{thm}[\cite{jia}] \label{thm:jia}
    	A 4-connected graph $G$ is $ K_{3,3}+v $-minor-free if and only if it is planar, $C_{2k+1}^{2}$ $(k \ge 2)$, $L(K_{3,3})$ or it is obtained from $C_{5}^{2}$ by repeatedly 4-splitting vertices.
    \end{thm}

	\begin{lem}[\cite{P7}]\label{lem2.1}
	The only 4-split of $C_{5}^{2}$ are $K_6$, $K_6$$\backslash$$e$, $DW_4$.
\end{lem}

Thus, we next characterize the $K_{3,3}+v$-minor-free graph obtained from $K_6$, $K_6$$\backslash$$e$, $DW_4$ by repeatedly 4-splitting vertices.

		\section{Proof of Theorem \ref{thm1.2}}
		In this section, we proof the following results from which the Theorem 1.1 follows. 
		
		\begin{lem}\label{lem3.1}
		The graph obtained by adding an edge to $\Gamma_1$ contains $ K_{3,3}+v $ as a minor.
		\end{lem}
	\begin{proof}
		By symmetry, we discuss three cases. The graph obtained by adding an edge $av_5$, $av_4$ or $v_1v_4$ to $\Gamma_1$ contains $ K_{3,3}+v $ as a minor (see Figure \ref{fig:04}).
	\end{proof}

\begin{figure}[!htb]
	\centering
	\includegraphics[width=0.7\linewidth]{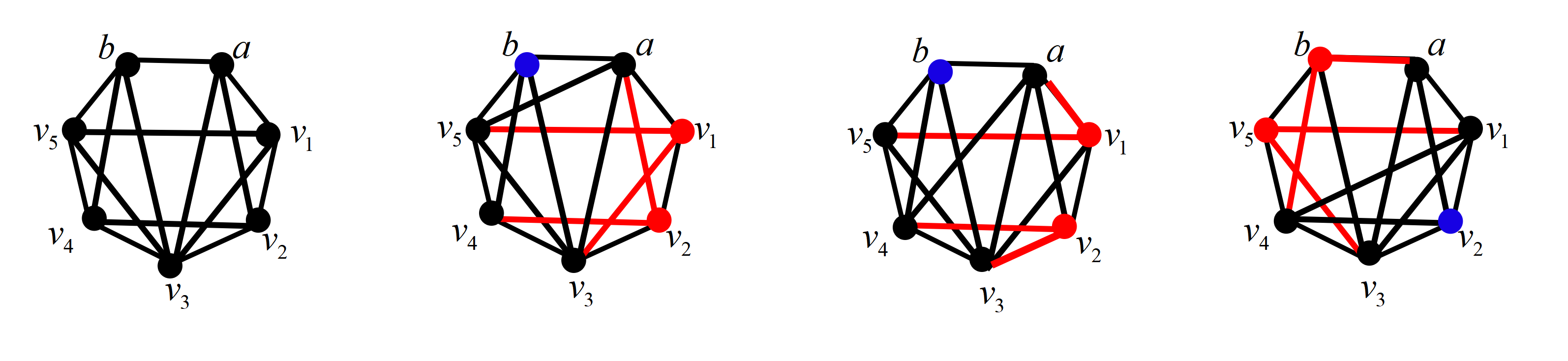}
	\caption{$\Gamma_1 $ and $\Gamma_1 + e$}
	\label{fig:04}
\end{figure}
	
	\begin{lem}\label{lem3.2}
The graph $\Gamma_1$ is the only $K_{3,3}+v$-minor-free graph obtained by 4-splitting a vertex of degree 5 in $DW_4$. 
\end{lem}
\begin{proof}
	Let \{$v_1,v_2,\ldots,v_6$\} be vertices of $DW_4$ (shown in Figure \ref{fig:02}). We first consider the minimal case for $|A|=|B|=3$, where $A, B$ are the subsets of the $N_G(v_6)$, $A \cup B$=$N_G(v_6)$. Suppose that $a, b$ are the new vertices obtained by 4-splitting $v_6$. Since the degree of $v_6$ is 5, we distinguish the discussion by the possibilities for $A$ $\cap$ \{$v_1, v_5$\} and $B$ $\cap$ \{$v_1, v_5$\}. The first case is that one of $A$ $\cap$ \{$v_1, v_5$\} and $B$ $\cap$ \{$v_1, v_5$\} is $\emptyset$. Without loss of generality, we assume that $B$ $\cap$ \{$v_1, v_5$\} is $\emptyset$. The second case is that $A$ $\cap$ \{$v_1, v_5$\} $\neq$ $\emptyset$ and $B$ $\cap$ \{$v_1, v_5$\} $ \neq $ $\emptyset$.
\begin{figure}[!htb]
	\centering
	\includegraphics[width=0.2\linewidth]{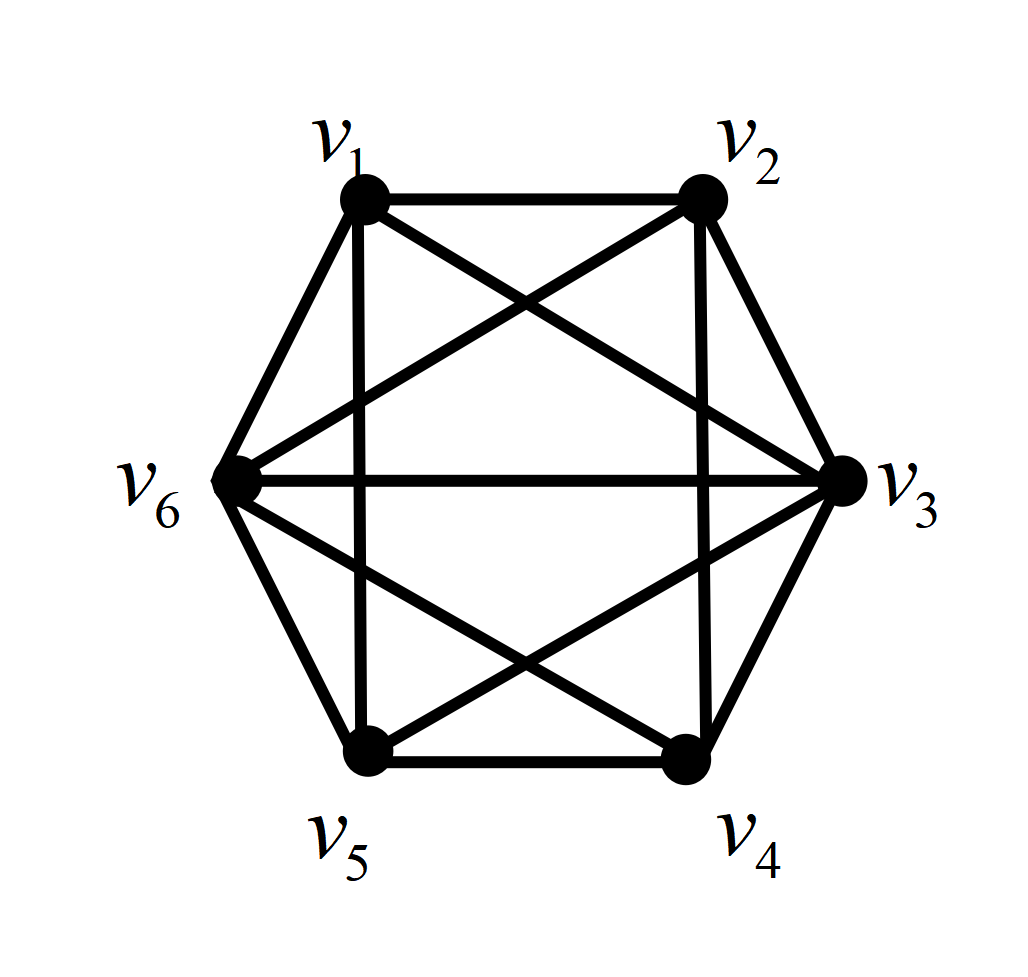}
	\caption{$DW_4$}
	\label{fig:02}
\end{figure}

	 {\bf Case 1.} $A$ $\cap$ $\{v_1, v_5\}$ $\neq$ $\emptyset$, $B$ $\cap$ $\{v_1, v_5\}$ $=$ $\emptyset$.
	
	 In this case, $B$ has at least three vertices. Thus, $B$=\{$v_2$, $v_3$, $v_4$\} and the set $A$ may be \{$v_1, v_5, v_2$\}, \{$v_1, v_5, v_3$\} or \{$v_1, v_5, v_4$\}. We list these three cases in Figure \ref{fig:03} respectively. It is not difficult to see that the second graph is a $ K_{3,3}+v $-minor-free graph which is isomorphic to $\Gamma_1$ and the other two graphs both contain $ K_{3,3}+v $ as a minor.
	 
\begin{figure}[!htb]
	\centering
	\includegraphics[width=0.5\linewidth]{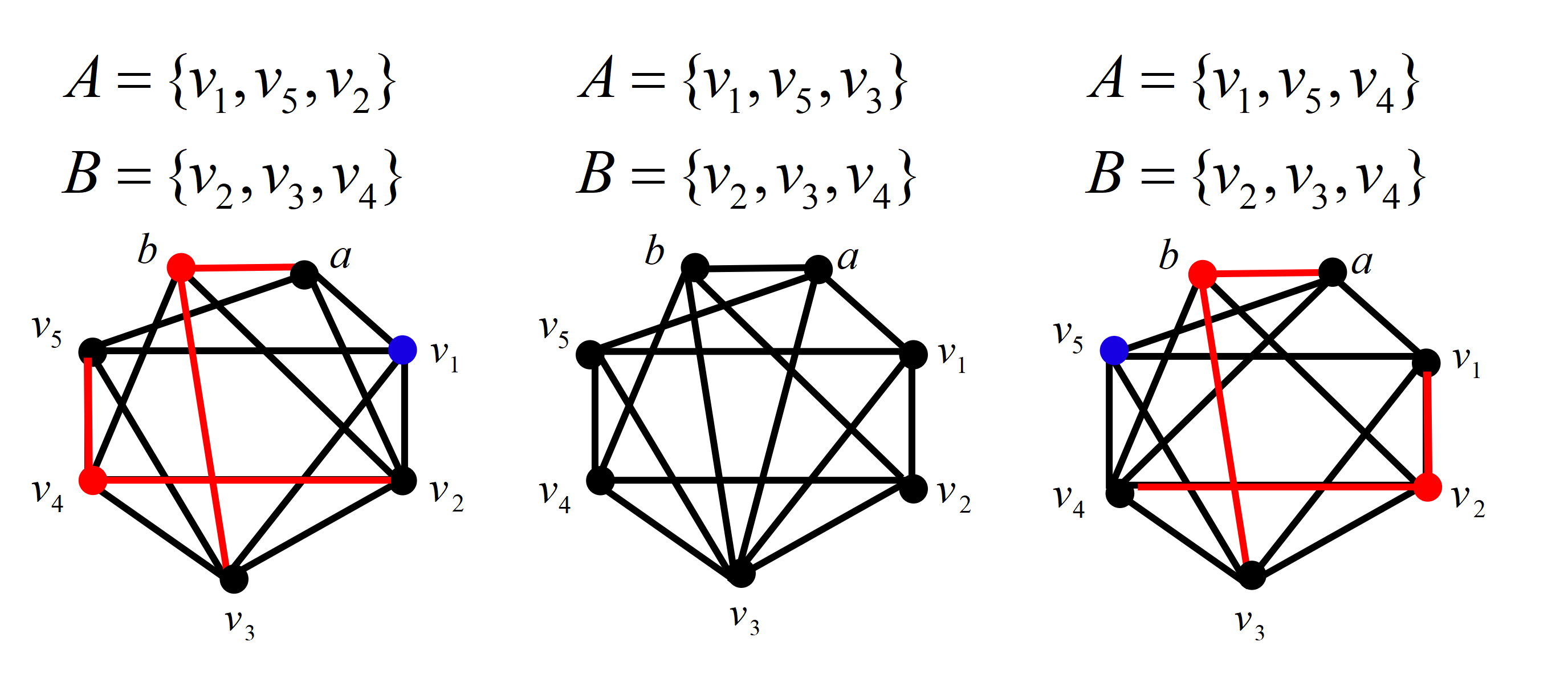}
	\caption{Three graphs in Case 1}
	\label{fig:03}
\end{figure}

  {\bf Case 2.}  $A$ $\cap$ $\{v_1, v_5\}$ $\neq$ $\emptyset$, $B$ $\cap$ $\{v_1, v_5\}$ $\neq$ $\emptyset$.
  
In this case, we assume that $v_1$ $\in$ $A$, $v_5$ $\in$ $B$ without loss of generality. Since $|N_G(v_6)|$ = 5, $A$ $\cap$ $B$ $\neq$ $\emptyset$ in the minimal case for $|A|=|B|=3$. By the symmertry of the graph $G$$\backslash$$v_6$, we distinguish the following three cases.
	 
 {\bf Case 2.1.}  $A \cap B $ = \{$v_2$\}.
	  
	 Since $v_1$$\in$$A$ and $v_2$$\in$$A$, one of $v_3$ and $v_4$ must belong to $A$. If $v_3$ $\in$ $A$, then $v_4$ must belong to $B$ in the minimal case. If $v_4$ $\in$ $A$, then $v_3$ must belong to $B$ in the minimal case. It is easy to check that both the two resulting graphs contain $ K_{3,3}+v $ as a minor (see Figure \ref{fig:05}).
	 
	  \begin{figure}[!htb]
	  	\centering
	  	\includegraphics[width=0.4\linewidth]{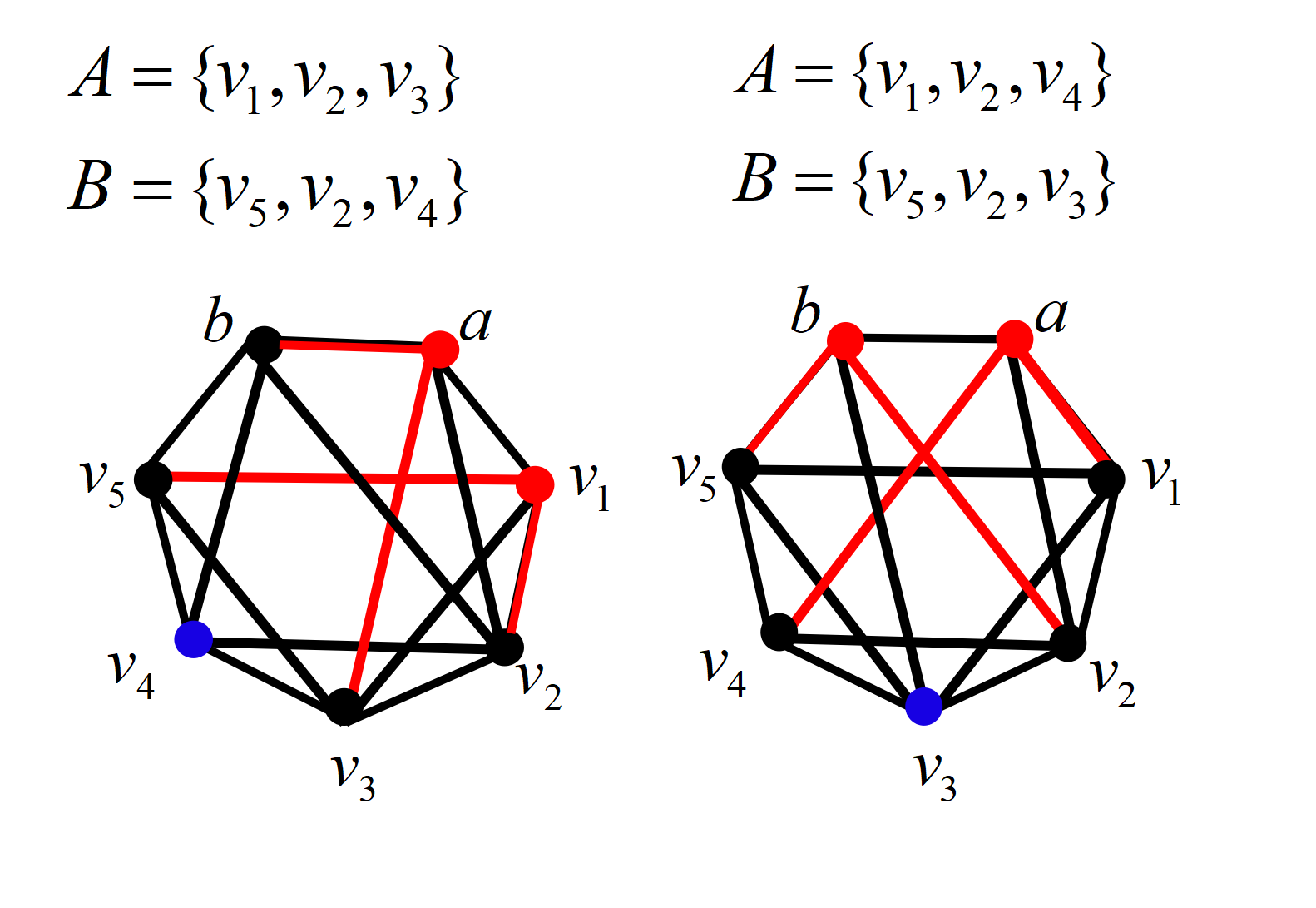}
	  	\caption{Two graphs of Case 2.1}
	  	\label{fig:05}
	  \end{figure}
  
	  {\bf Case 2.2.}  $A \cap B $ = \{$v_5$\}.
	  
	  	 Since $v_1$$\in$$A$, $v_5$$\in$$A$, one of $v_2$, $v_3$ and $v_4$ must belong to $A$. First, if $v_2$ $\in$ $A$, then $v_3$ and $v_4$ must belong to $B$ in the minimal case. Second, if $v_3$ $\in$ $A$, then $v_2$ and $v_4$ must belong to $B$ in the minimal case. Third, if $v_4$ $\in$ $A$, then $v_2$ and $v_3$ must belong to $B$ in the minimal case. It is straightforward to verify that all the three resulting graphs contain $ K_{3,3}+v $ as a minor (see Figure \ref{fig:06}).
	  	 
	  \begin{figure}[!htb]
	  	\centering
	  	\includegraphics[width=0.5\linewidth]{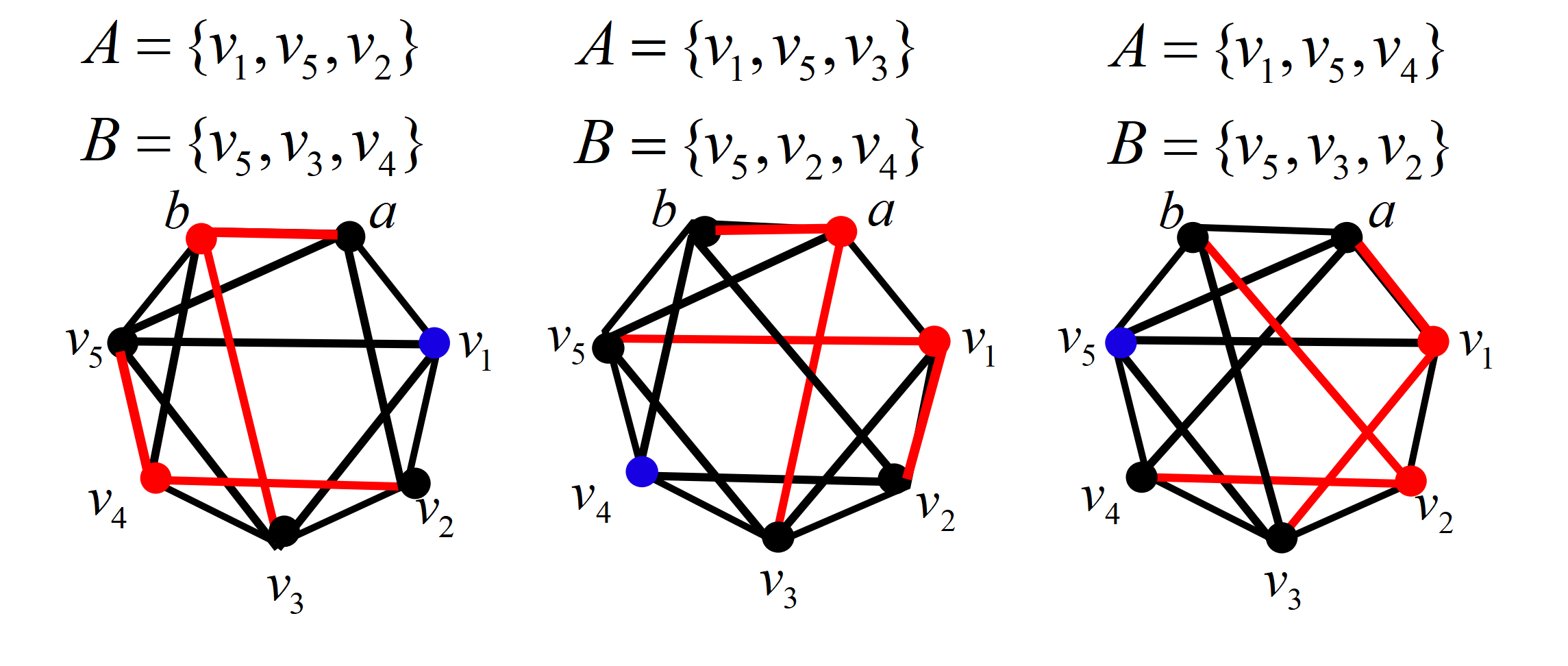}
	  	\caption{Three graphs of Case 2.2}
	  	\label{fig:06}
	  \end{figure}

	   {\bf Case 2.3.}  $A \cap B $ = \{$v_3$\}.
	   
	    The process of discussing the sets of A and B is the same as the Case 2.1. Hence we can obtain two graphs as shown in Figure \ref{fig:07}. It is clear that the first graph is $\Gamma_1$ which is a $ K_{3,3}+v $-minor-free graph and the second graph contains $K_{3,3}+v$ as a minor.

	\begin{figure}[!htb]
		\centering
		\includegraphics[width=0.4\linewidth]{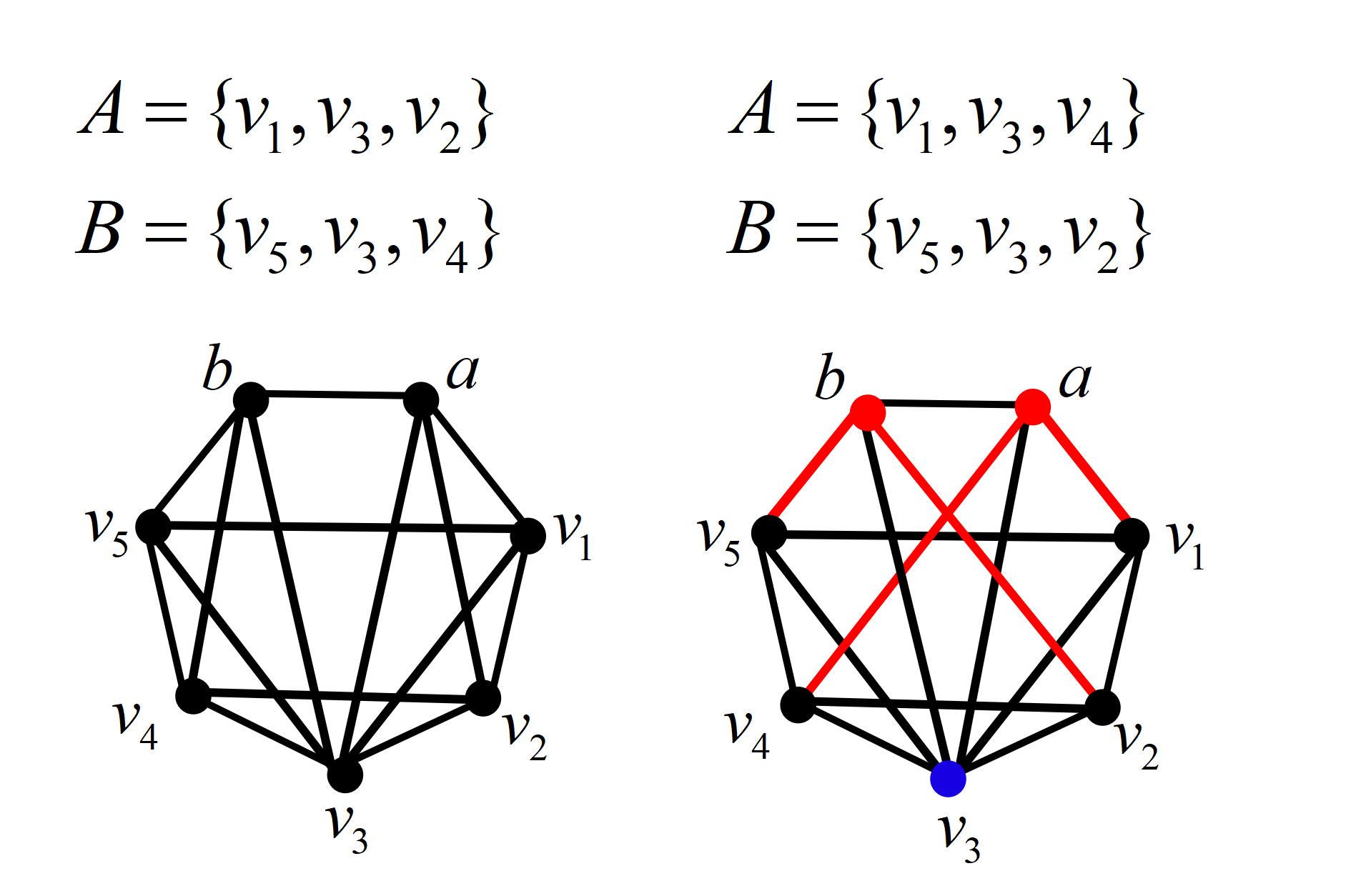}
		\caption{Two graphs of Case 2.3}
		\label{fig:07}
	\end{figure}
	
	Besides $\Gamma_1$, the new graph $G^{'}$ obtained by 4-splitting the vertex $v_6$ from $DW_4$ with $|A|$ = $|B|$ = 3 contains $ K_{3,3}+v $ as a minor. The graph obtained in non-minimal cases is equivalent to add new edges $av_i$ or $bv_j$ to the above ten graphs mentioned in Figure \ref{fig:03}--\ref{fig:07}. Then by Lemma \ref{lem3.1} and the result of the above ten graphs, all the other new graphs obtained by 4-splitting $v_6$ from $DW_4$ contain $ K_{3,3}+v $ as a minor.	
	
\end{proof}

	\begin{lem}\label{lem3.3}
The graph obtained by 4-splitting a vertex of degree 4 in $\Gamma_1$ contains $ K_{3,3}+v $ as a minor. 
	\end{lem}
	
	\begin{proof}
		There are six vertices of degree 4 in $\Gamma_1$. Up to symmetry, we consider three cases. We first consider the minimal case for $|A|=|B|=3$, where $A, B$ are the subsets of the $N_G(v)$, $A \cup B$=$N_G(v)$. Let \{$v_1,v_2,\ldots,v_7$\} be vertices of $\Gamma_1$ (shown in Figure \ref{fig:08}).
		
		{\bf Case 1.}  Split the vertex $v_1$.
		
		Since $|N_G(v_1)|$ = 4, $A$ $\cap$ $\{v_2, v_7\}$ $\neq$ $\emptyset$ and $B$ $\cap$ $\{v_2, v_7\}$ $\neq$ $\emptyset$ in the minimal case for $|A|=|B|=3$. Without loss of generality, we assume that $v_2$ $\in$ $A$, $v_7$ $\in$ $B$. The set $A$ may  be \{$v_2, v_3, v_4$\}, \{$v_2, v_3, v_7$\} or \{$v_2, v_4, v_7$\}. The set $B$ may be \{$v_2, v_3, v_7$\}, \{$v_2, v_4, v_7$\} or \{$v_3, v_4, v_7$\}. On the one hand, if the set $A$ is \{$v_2, v_3, v_4$\}, then the set $B$ may be \{$v_2, v_3, v_7$\}, \{$v_2, v_4, v_7$\} or \{$v_3, v_4, v_7$\}. On the other hand, if the set $A$ is \{$v_2, v_3, v_7$\}, then the set $B$ may be \{$v_2, v_4, v_7$\} or \{$v_3, v_4, v_7$\}. At last if the set $A$ is \{$v_2, v_4, v_7$\}, then the set $B$ may be \{$v_2, v_3, v_7$\} or \{$v_3, v_4, v_7$\}. Hence we have seven graphs. It is evident that all the seven resulting graphs contain $ K_{3,3}+v $ as a minor (see Figure \ref{fig:08}).
		
\begin{figure}[!htb]
	\centering
	\includegraphics[width=0.7\linewidth]{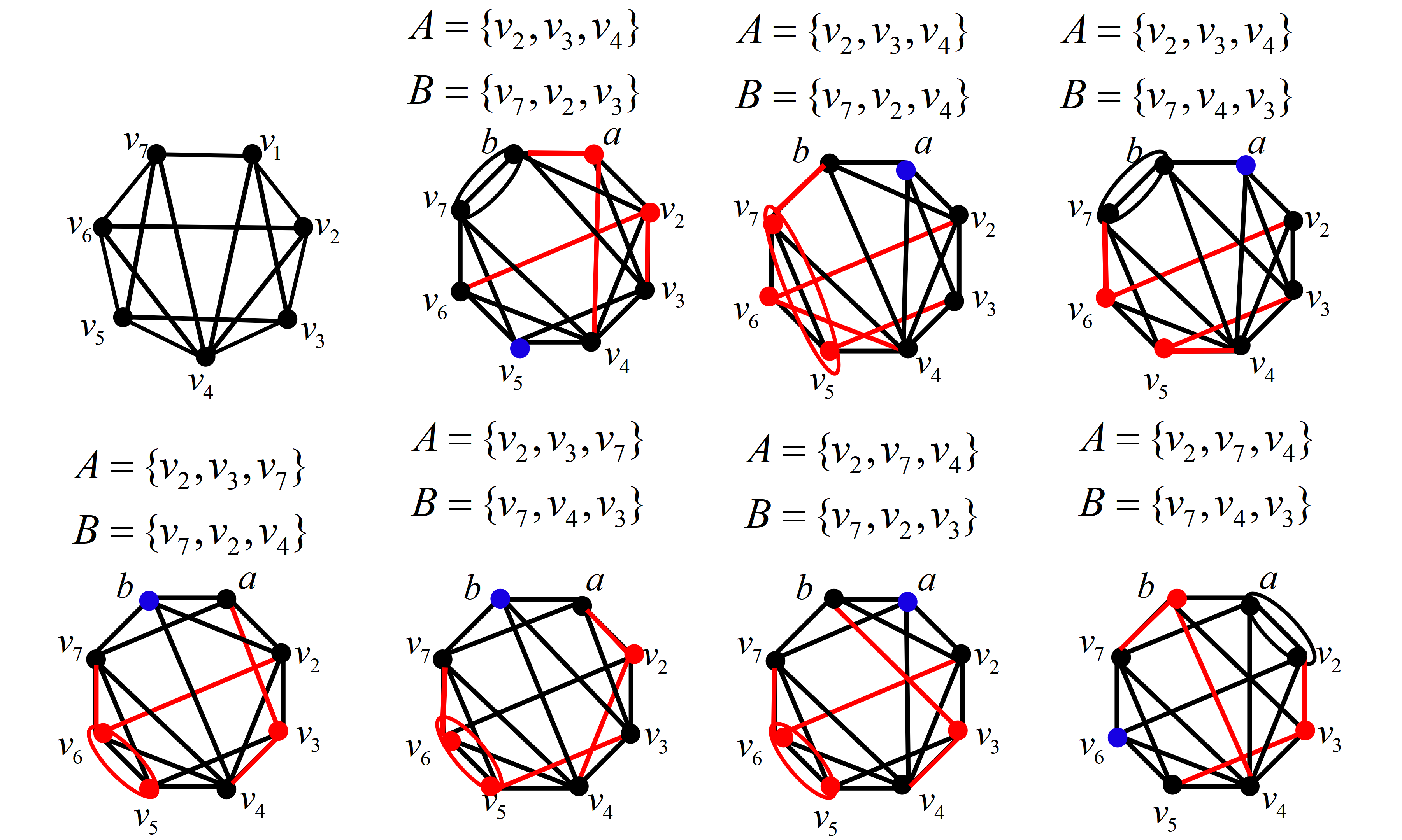}
	\caption{$\Gamma_1$ and seven graphs obtained by minimal 4-splitting $v_1$}
	\label{fig:08}
\end{figure}

	{\bf Case 2.}  Split the vertex $v_2$.
	
		Since $|N_G(v_2)|$ = 4, $A$ $\cap$ $\{v_1, v_3\}$ $\neq$ $\emptyset$ and $B$ $\cap$ $\{v_1, v_3\}$ $\neq$ $\emptyset$ in the minimal case for $|A|=|B|=3$. Without loss of generality, we assume that $v_3$ $\in$ $A$, $v_1$ $\in$ $B$. The set $A$ may  be \{$v_1, v_3, v_4$\}, \{$v_1, v_3, v_6$\} or \{$v_3, v_4, v_6$\}. The set $B$ may be \{$v_1, v_3, v_4$\}, \{$v_1, v_3, v_6$\} or \{$v_1, v_4, v_6$\}. Similarly to the previous case, we have seven graphs. It is easy to check that all the seven resulting graphs contain $ K_{3,3}+v $ as a minor (see Figure \ref{fig:09}).
	
\begin{figure}[!htb]
	\centering
	\includegraphics[width=0.7\linewidth]{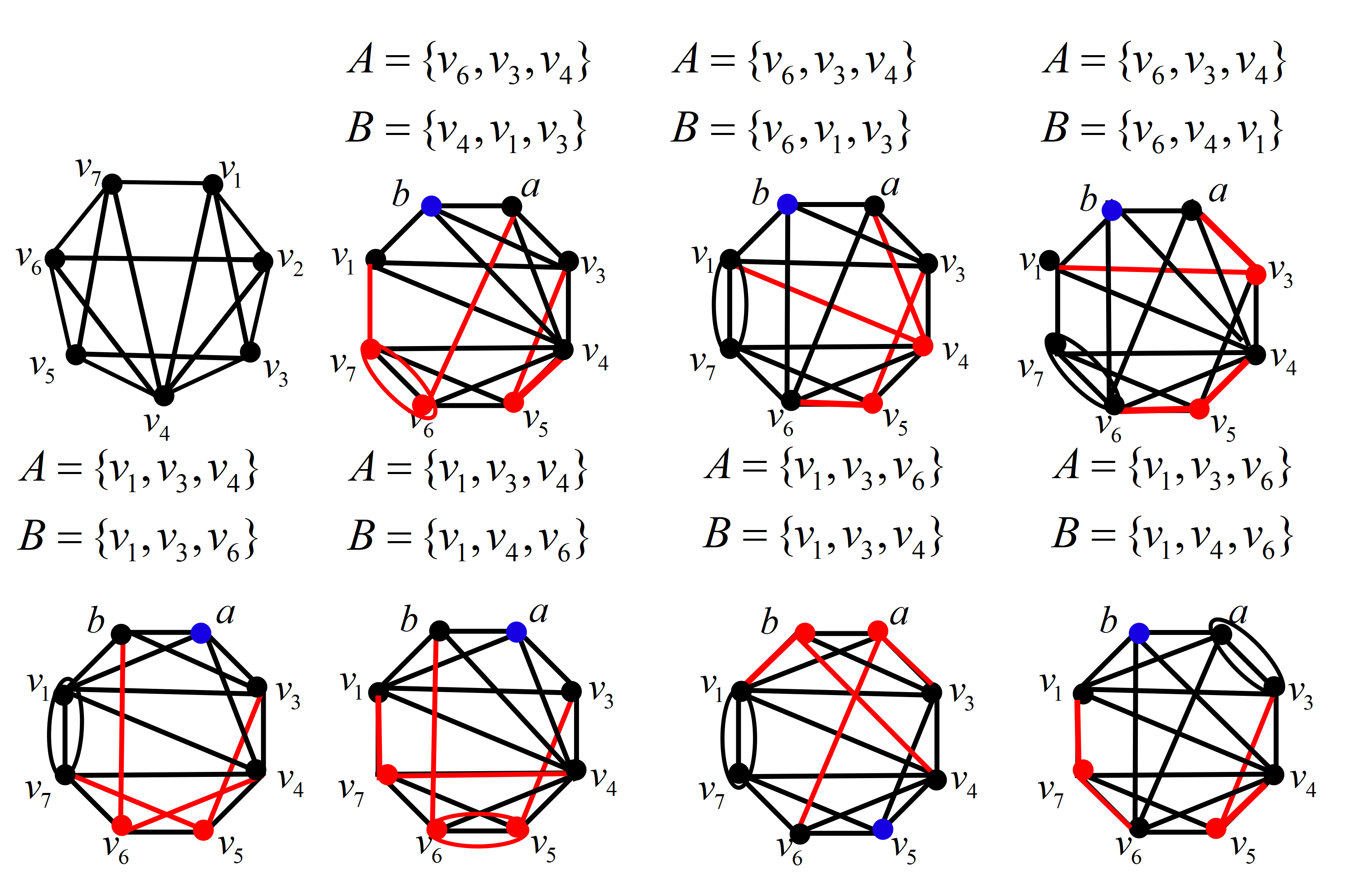}
	\caption{$\Gamma_1$ and seven graphs obtained by minimal 4-splitting $v_2$}
	\label{fig:09}
\end{figure}

		{\bf Case 3.}  Split the vertex $v_3$.
		
		Since $|N_G(v_3)|$ = 4, $A$ $\cap$ $\{v_2, v_4\}$ $\neq$ $\emptyset$ and $B$ $\cap$ $\{v_2, v_4\}$ $\neq$ $\emptyset$ in the minimal case for $|A|=|B|=3$. Without loss of generality, we assume that $v_4$ $\in$ $A$, $v_2$ $\in$ $B$. The set $A$ may  be \{$v_1, v_4, v_5$\}, \{$v_2, v_4, v_5$\} or \{$v_1, v_2, v_4$\}. The set $B$ may be \{$v_2, v_4, v_5$\}, \{$v_1, v_2, v_4$\} or \{$v_1, v_2, v_5$\}. By an argument similar to that of Case 1, we have seven graphs. It is clear that all the seven resulting graphs contain $ K_{3,3}+v $ as a minor (see Figure \ref{fig:10}).
		
\begin{figure}[!htb]
	\centering
	\includegraphics[width=0.7\linewidth]{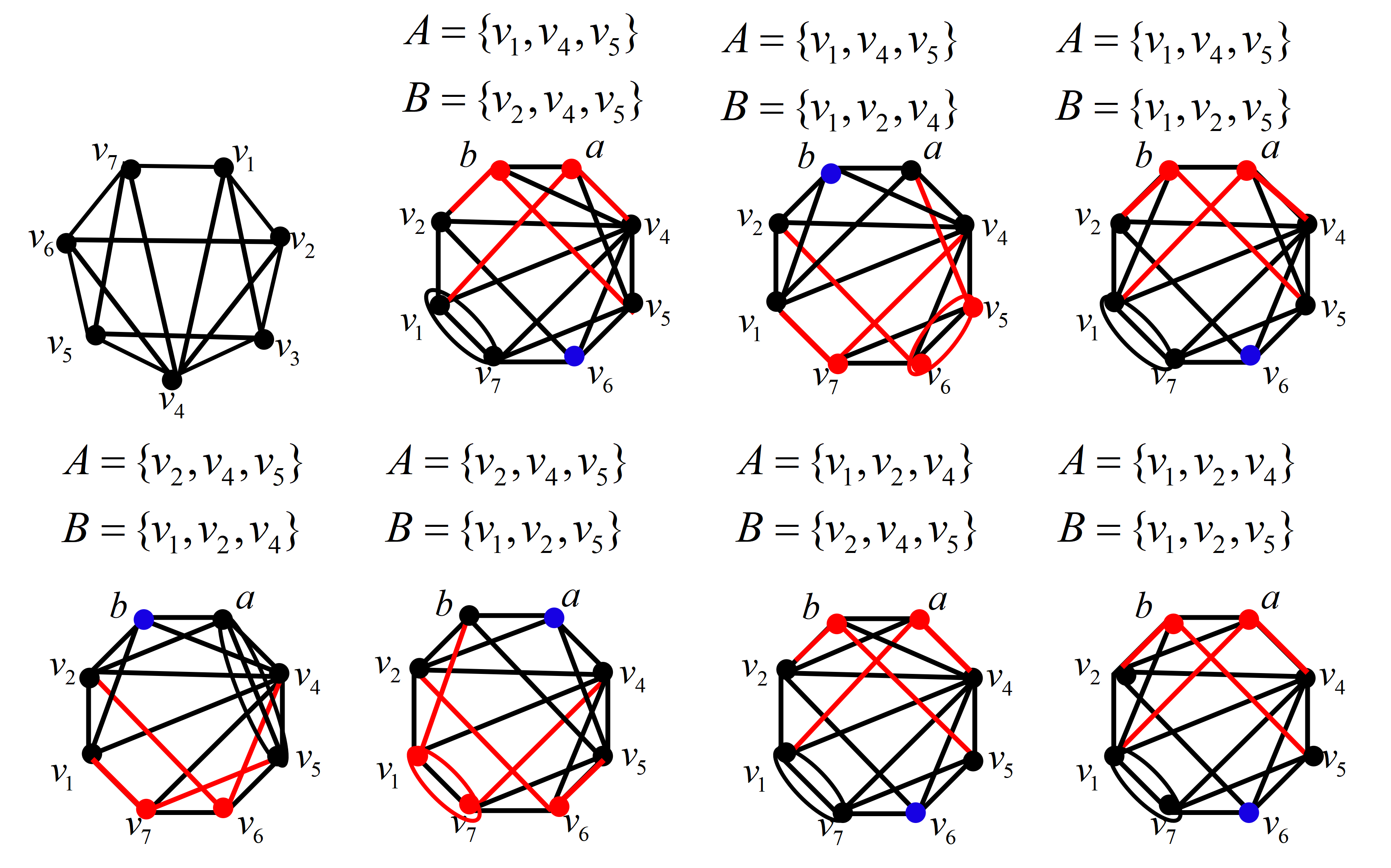}
	\caption{$\Gamma_1$ and seven graphs obtained by minimal 4-splitting $v_3$}
	\label{fig:10}
\end{figure}

		The graph obtained in non-minimal cases is equivalent to add new edges $av_i$ or $bv_j$ to the above  graphs mentioned in Figures \ref{fig:08}--\ref{fig:10}. Hence, the graph obtained by 4-splitting a vertex of degree 4 in $\Gamma_1$ contains $K_{3,3}+v$ as a minor.
		
		\end{proof}
	
		\begin{lem}\label{lem3.4}
		 The graph obtained by adding a new edge to $\Gamma_3$ contains $ K_{3,3}+v $ as a minor.
		
	\end{lem}

\begin{proof}
	As shown in Figure \ref{fig:14}, it can be seen that $\Gamma_2 + av_3 $, $\Gamma_2 + av_2 $, $\Gamma_2 + v_5v_1 $ and $\Gamma_2 + v_5v_2 $ contain $ K_{3,3}+v $ as a minor, and $\Gamma_2 + av_1 $ is $\Gamma_3$ which is a $ K_{3,3}+v $-minor-free graph. Similarly, $\Gamma_3 + bv_7$ contains $ K_{3,3}+v $ as a minor (see Figure \ref{fig:15}). Thus, the graph obtained by adding a new edge to $\Gamma_3$ has a $ K_{3,3}+v$ minor.
	
	\begin{figure}[!htb]
		\centering
		\includegraphics[width=0.7\linewidth]{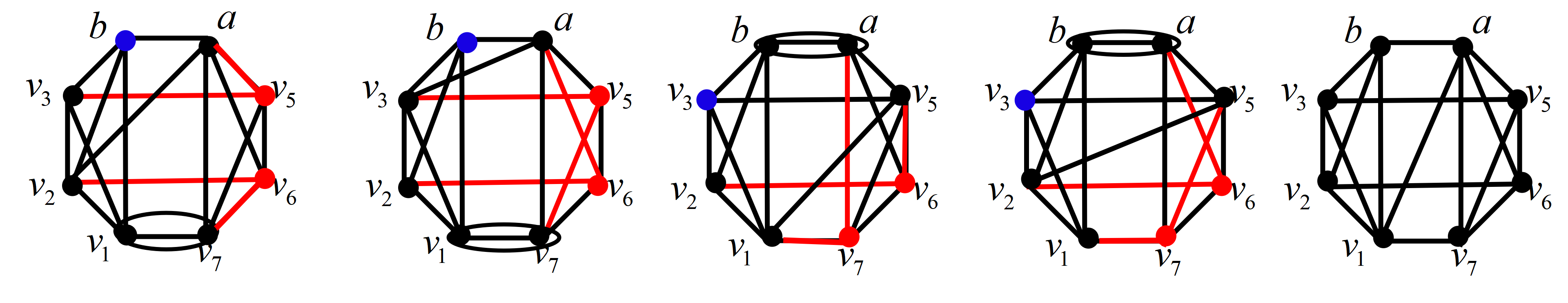}
		\caption{$\Gamma_2 + av_2 $, $\Gamma_2 + av_3 $,  $\Gamma_2 + v_5v_1 $, $\Gamma_2 + v_5v_2 $, $\Gamma_3$ }
		\label{fig:14}
	\end{figure}

\begin{figure}[!htb]
	\centering
	\includegraphics[width=0.15\linewidth]{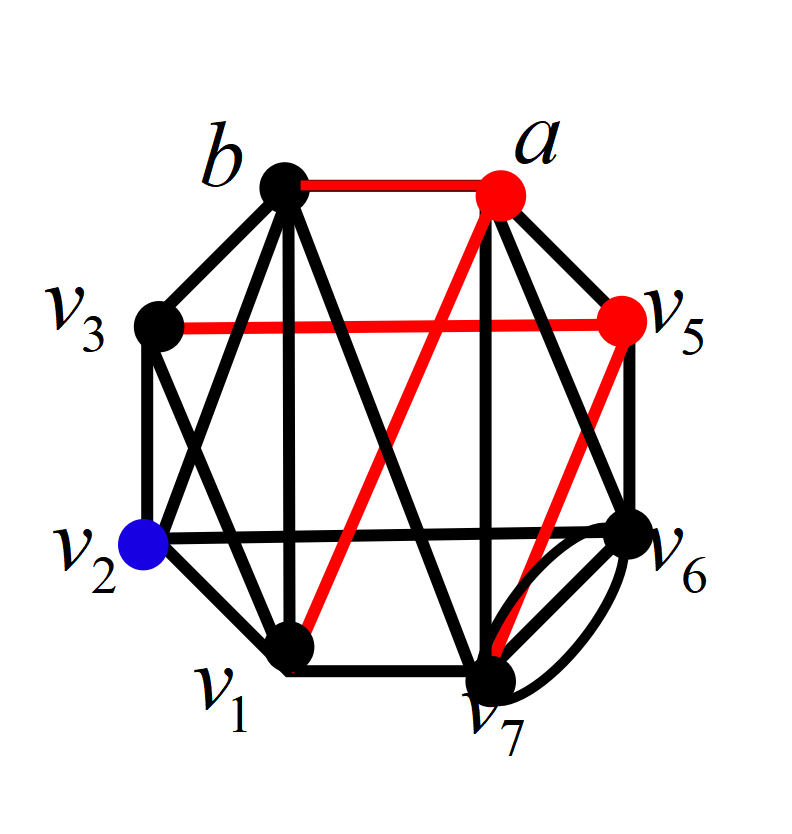}
	\caption{$\Gamma_3 + bv_7$}
	\label{fig:15}
\end{figure}

\end{proof}

			\begin{lem}\label{lem3.5}
				The graph obtained by 4-splitting a vertex of degree 6 in $\Gamma_1$ contains $K_{3,3}+v$ as a minor besides $\Gamma_2$ and $\Gamma_3$, .
		\end{lem}
		\begin{proof}
			
		Let \{$v_1,v_2,\ldots,v_7$\} be vertices of $\Gamma_1$ (shown in Figure \ref{fig:11}). We first consider the minimal case for $A$ and $B$, where $A, B$ are the subsets of the $N_G(v_4)$, $A \cup B$=$N_G(v_4)$. Suppose that $a, b$ are the new vertices obtained by 4-splitting the vertex $v_4$. The degree of $v_4$ is 6, thus there are three cases for $A$ and $B$. The first case is that one of $A$ $\cap$ \{$v_3, v_5$\} and $B$ $\cap$ \{$v_3, v_5$\} is $\emptyset$. Without loss of generality, we assume that $B$ $\cap$ \{$v_3, v_5$\} is $\emptyset$. The second case is that $A$ $\cap$ \{$v_3, v_5$\} $\neq$ $\emptyset$, $B$ $\cap$ \{$v_3, v_5$\} $\neq$ $\emptyset$ and $A$ $\cap$ $B$ $=$ $\emptyset$. The third case is that $A$ $\cap$ \{$v_3, v_5$\} $\neq$ $\emptyset$, $B$ $\cap$ \{$v_3, v_5$\} $\neq$ $\emptyset$ and $A$ $\cap$ $B$ $\neq$ $\emptyset$.
			
		 {\bf Case 1.}  $B$ $\cap$ \{$v_3, v_5$\} = $\emptyset$.
			 
		We consider the minimal case for $|A|=|B|=3$ and $B$ $\cap$ \{$v_3, v_5$\} = $\emptyset$. If the set $A$ is \{$v_3, v_5, v_1$\}, then the set $B$ is \{$v_2, v_6, v_7$\} in the minimal case. Similarly if the set $A$ is \{$v_3, v_5, v_2$\}, \{$v_3, v_5, v_6$\} or \{$v_3, v_5, v_7$\}, then the set $B$ is \{$v_1, v_6, v_7$\}, \{$v_2, v_1, v_7$\} or \{$v_2, v_6, v_1$\} respectively. All the four resulting graphs contain $ K_{3,3}+v $ as a minor (see Figure \ref{fig:11}).
		 
			\begin{figure}[!htb]
				\centering
				\includegraphics[width=0.7\linewidth]{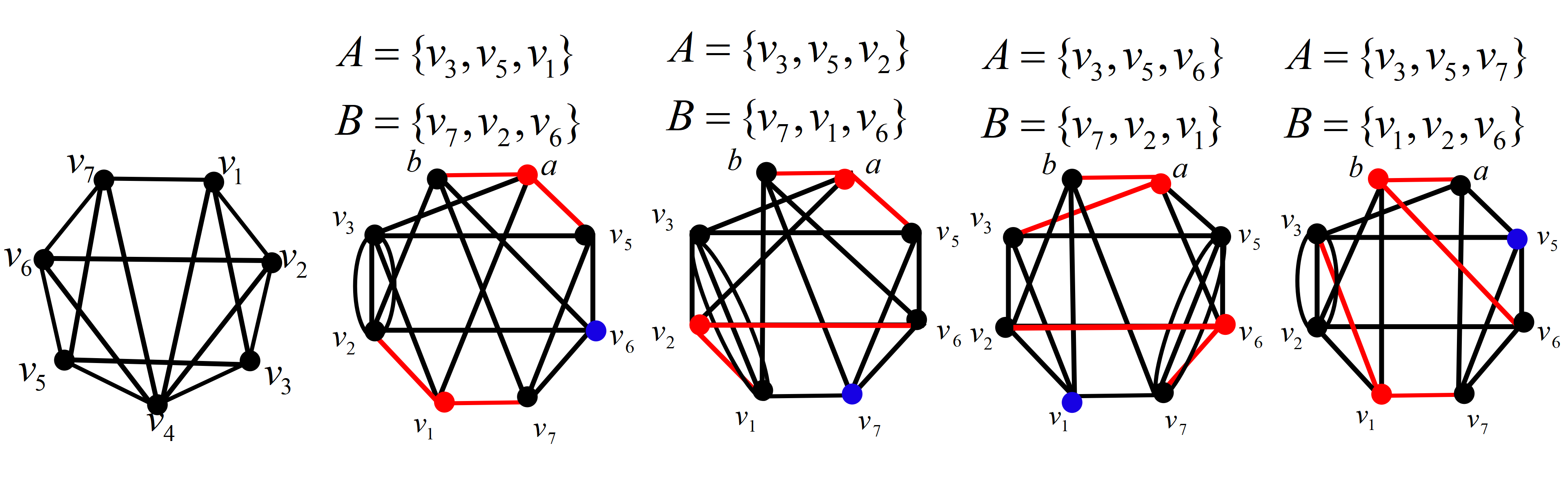}
				\caption{$\Gamma_1$ and four graphs of Case 1}
				\label{fig:11}
			\end{figure}
		
		 {\bf Case 2.}  $A$ $\cap$ \{$v_3, v_5$\} $\neq$ $\emptyset$, $B$ $\cap$ \{$v_3, v_5$\} $\neq$ $\emptyset$ and $A$ $\cap$ $B$ $=$ $\emptyset$
		 
				We consider the minimal case for $|A|=|B|=3$ and assume that $v_5$$\in$$A$, $v_3$$\in$$B$ without loss of generality. If the set $A$ is \{$v_1, v_2, v_5$\}, then the set $B$ is \{$v_3, v_6, v_7$\} in the minimal case. Similarly if the set $A$ is \{$v_1, v_5, v_6$\}, \{$v_1, v_5, v_7$\}, \{$v_2, v_5, v_6$\}, \{$v_2, v_5, v_7$\} or \{$v_5, v_6, v_7$\}, then the set $B$ is \{$v_2, v_3, v_7$\}, \{$v_2, v_3, v_6$\}, \{$v_3, v_1, v_7$\}, \{$v_3, v_1, v_6$\} or \{$v_1, v_2, v_3$\} respectively. It is easy to check that all the
				resulting graphs contain $ K_{3,3}+v $ as a minor besides the sixth graph $\Gamma_2$ (see Figure \ref{fig:12}).
		 
		 \begin{figure}[!htb]
		 	\centering
		 	\includegraphics[width=0.5\linewidth]{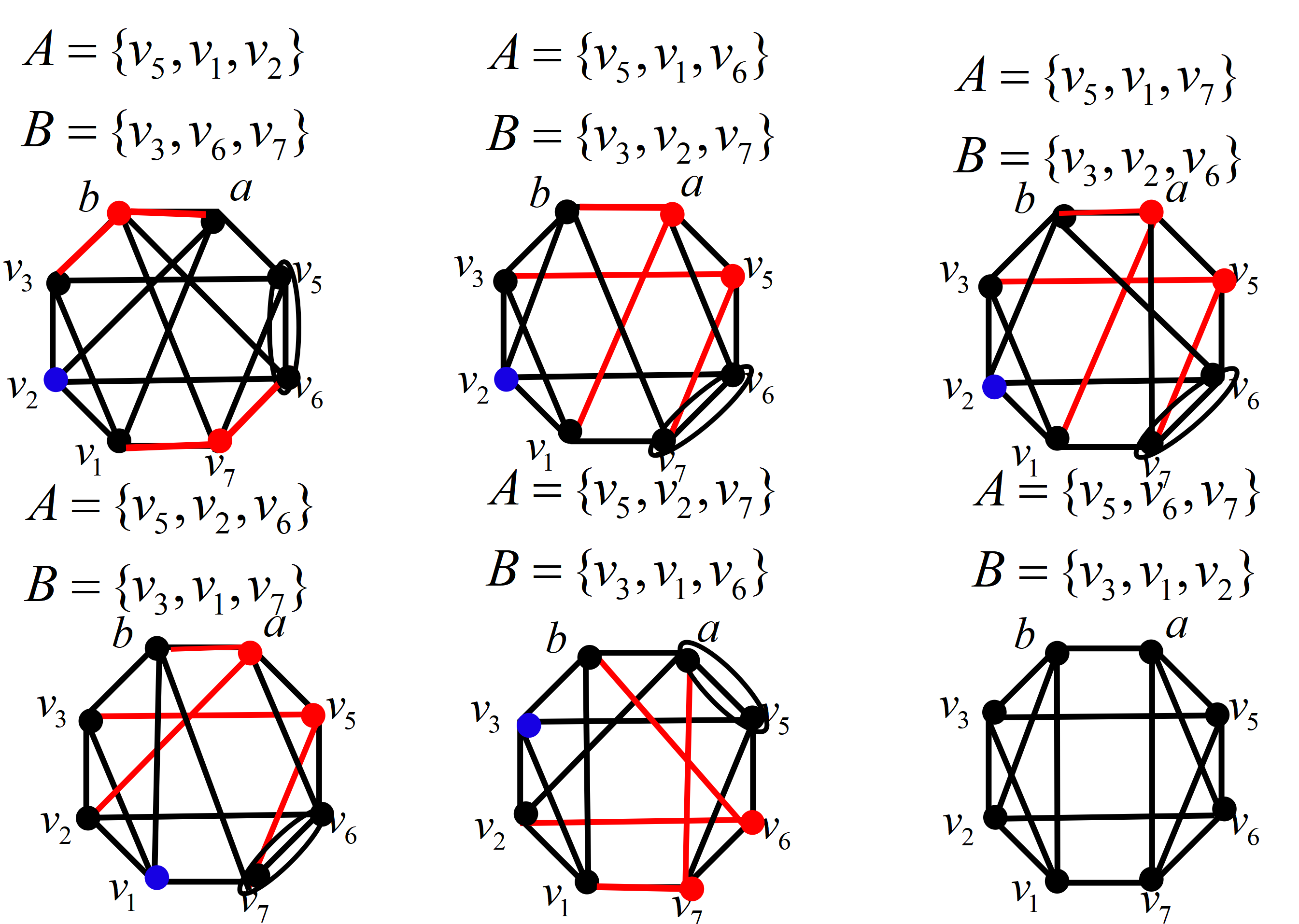}
		 	\caption{Six graphs of Case 2}
		 	\label{fig:12}
		 \end{figure}
	 
	  {\bf Case 3.}  $A$ $\cap$ \{$v_3, v_5$\} $\neq$ $\emptyset$, $B$ $\cap$ \{$v_3, v_5$\} $\neq$ $\emptyset$ and $A$ $\cap$ $B$ $\neq$ $\emptyset$.
	   
	   {\bf Case 3.1.}  $A$ $\cap$ $B$ $=$ $\{v_1\}$
	  
	  Since $|N_G(v_4)|$=6 and $A$ $\cap$ $B$ $=$ $\{v_1\}$, the set $A$ or set $B$ contains 4 vertices in the minimal case. The graph generated in Case 3.1 is equivalent to add a new edge $av_1$ or $bv_1$ to the above graphs mentioned in Case 2. Thus, the graphs in Case 3.1 contain $ K_{3,3}+v $ as a minor besides $\Gamma_3$. The cases of $A$ $\cap$ $B$ $=$ $\{v_2\}$, $A$ $\cap$ $B$ $=$ $\{v_6\}$ and $A$ $\cap$ $B$ $=$ $\{v_7\}$ are similar to the above.
	  
	   {\bf Case 3.2.}  $A$ $\cap$ $B$ $=$ $\{v_3\}$
	  
	  Since $|N_G(v_4)|$=6 and $A$ $\cap$ $B$ $=$ $\{v_3\}$, the set $A$ or set $B$ contains 4 vertices in the minimal case. The graph generated in Case 3.2 is equivalent to add a new edge $av_3$ or $bv_3$ to the above graphs mentioned in Case 1 and Case 2. Thus, the graphs in Case 3.2 all contain $ K_{3,3}+v $ as a minor according to the Lemma \ref{lem3.4} and the results of the Case 1 and Case 2. The case of $A$ $\cap$ $B$ $=$ $\{v_5\}$ is similar to the above.
	 
	   In the other cases that are not minimal, it is equivalent to add new edges $av_i$ or $bv_j$ to the above graphs mentioned in Figure \ref{fig:11} and \ref{fig:12}. According to Lemma \ref{lem3.4} and above results, the $ K_{3,3}+v $-minor-free graphs are $\Gamma_2$ and $\Gamma_3$ by 4-splitting a vertex of degree 6 in $\Gamma_1$. 
	   
	   \end{proof}
	   
	   \begin{lem}\label{lem3.6}  	
	   Every 4-split of  $\Gamma_2$ contains a $ K_{3,3}+v $-$ minor$.
	   \end{lem}
   
   \begin{proof}
   	We consider the minimal case for $|A|=|B|=3$, where $A, B$ are the subsets of the $N_G(v)$, $A \cup B$=$N_G(v)$. Let \{$a,b,v_1,v_2,v_3,v_5,v_6,v_7$\} be vertices of $\Gamma_2$ (shown in Figure \ref{fig:16}). By symmetry, we split the vertex $a$ and the vertex $v_5$. Note that $d(a)=4$ and $d(v_5)=4$. Thus the process of 4-splitting the vertex $a$ or $v_5$ is similar to that of Lemma \ref{lem3.3}.
   	
   	{\bf Case 1.}  Split the vertex $a$.
   	
   	$A \cap \{b, v_5\}$ $\neq$ $\emptyset$, $B \cap \{b, v_5\}$ $\neq $ $\emptyset$ in the minimal case for $|A|=|B|=3$. Without loss of generality, we assume that $v_5$ $\in$ $A$, $b$ $\in$ $B$. We have seven graphs and it is clear that all the seven resulting graphs contain $ K_{3,3}+v $ as a minor (see Figure \ref{fig:16}).
   	
\begin{figure}[!htb]
	\centering
	\includegraphics[width=0.7\linewidth]{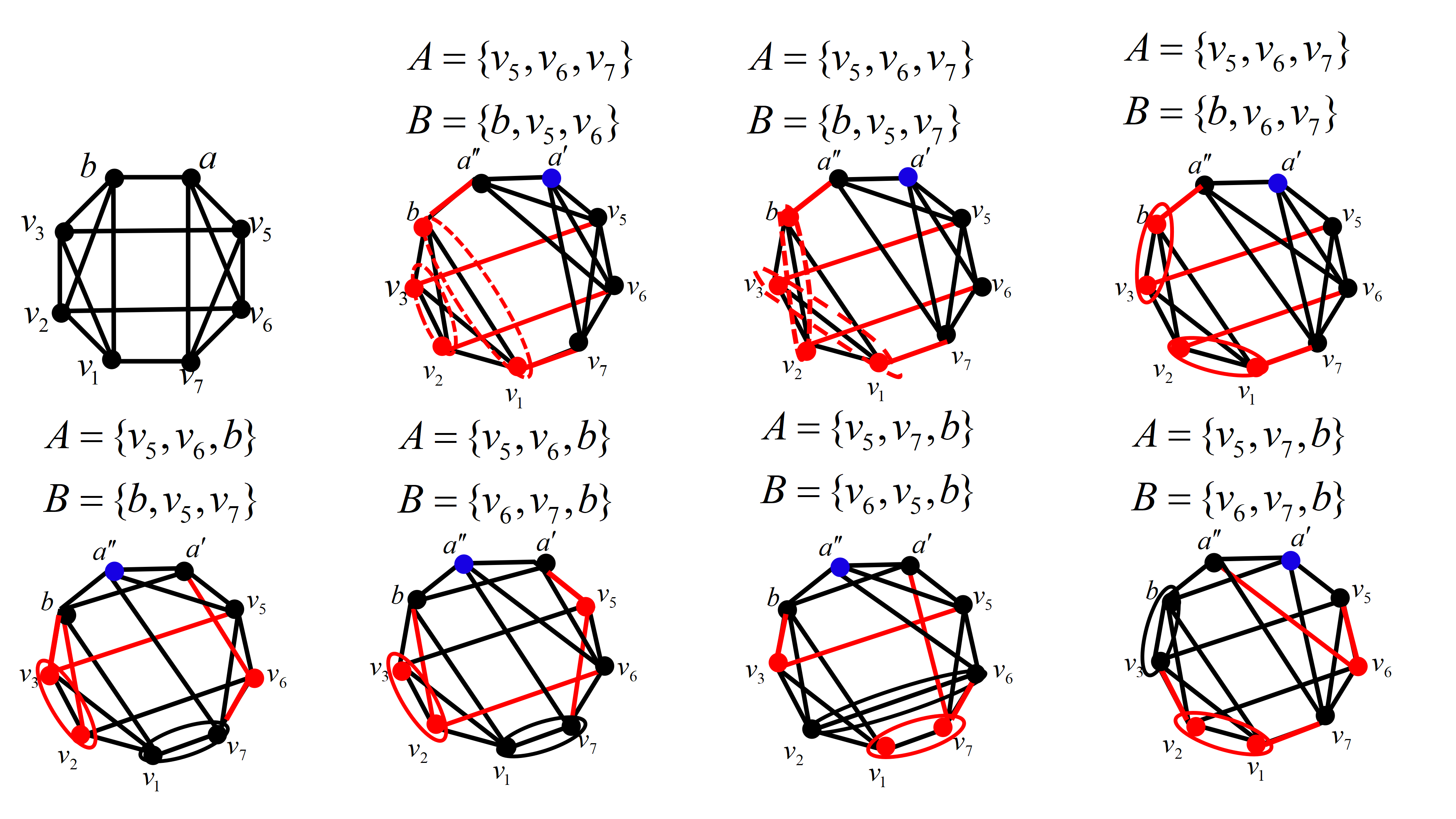}
	\caption{$\Gamma_2$ and seven graphs of Case 1}
	\label{fig:16}
\end{figure}

 	{\bf Case 2.}  Split the vertex $v_5$.
   		
   	$A \cap \{a, v_6\}$ $\neq$ $\emptyset$, $B \cap \{a, v_6\}$ $\neq $ $\emptyset$ in the minimal case for $|A|=|B|=3$. Without loss of generality, we assume that $v_6$ $\in$ $A$, $a$ $\in$ $B$. We obtain seven graphs which contain $K_{3,3}+v$ as a minor (see Figure \ref{fig:17}).
   		
   		\begin{figure}[!htb]
   			\centering
   			\includegraphics[width=0.7\linewidth]{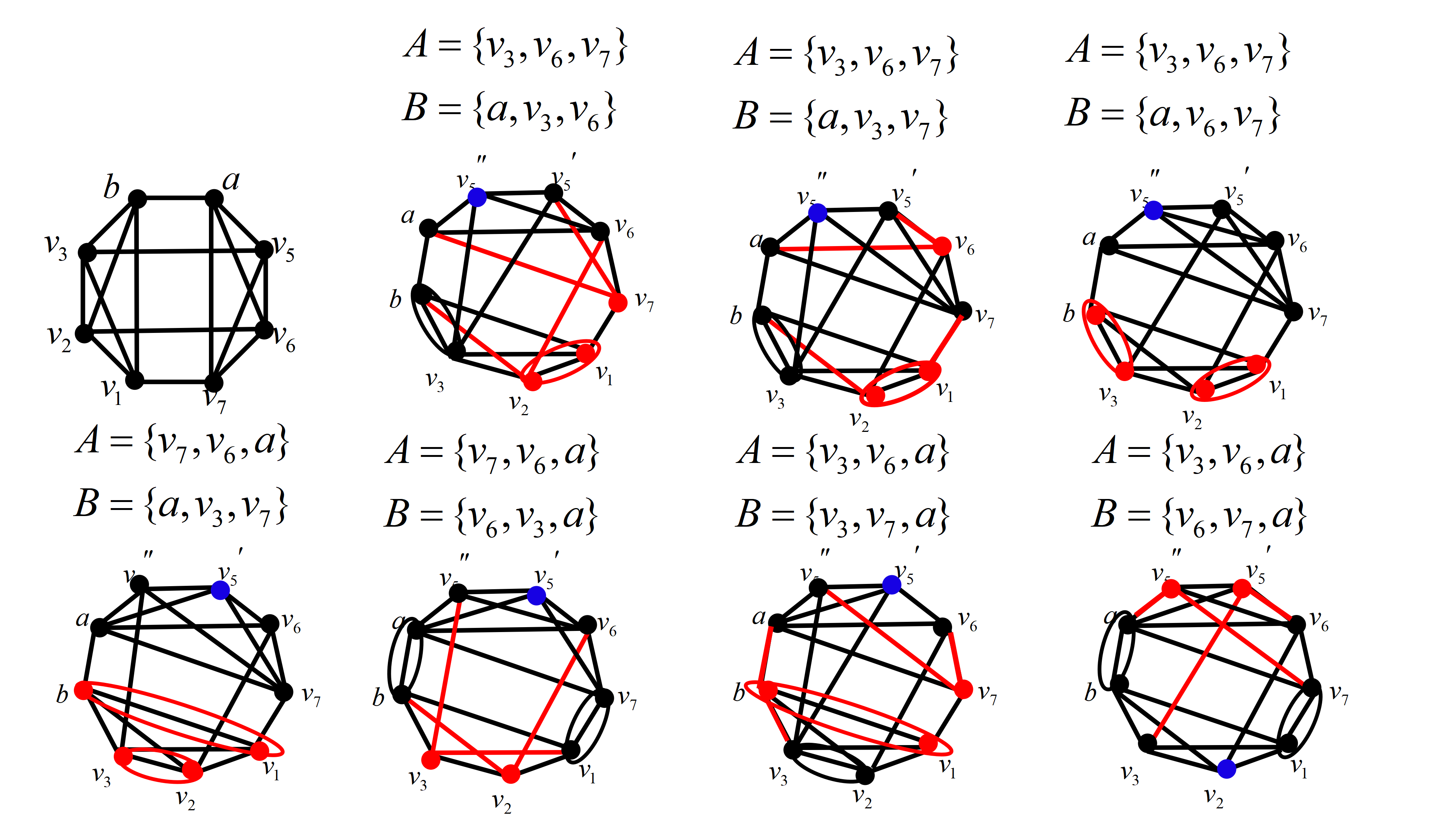}
   			\caption{$\Gamma_2$ and seven graphs of Case 2}
   			\label{fig:17}
   		\end{figure}
   	
   	In all minimal cases, the new graph $G^{'}$ generated by 4-splitting vertice from $\Gamma_2$ with $ |A| = |B| =3 $ contains  $ K_{3,3}+v $. Clearly, all the other graphs obtained by 4-splitting a vertex of degree 4 in $\Gamma_2$ contain  $ K_{3,3}+v $ as a minor.		
	   
		\end{proof}
	
	 \begin{lem}\label{lem3.7}
		Every 4-split of $\Gamma_3$ contains a $ K_{3,3}+v $-$ minor$.
		
	\end{lem}
	
	\begin{proof}
	We consider the minimal case for $|A|=|B|=3$, where $A, B$ are the subsets of the $N_G(v)$, $A \cup B$=$N_G(v)$. Let \{$a,b,v_1,v_2,v_3,v_5,v_6,v_7$\} be vertices of $\Gamma_3$ (shown in Figure \ref{fig:18}). It appears obvious that the graph obtained by splitting a vertex of degree 4 in $\Gamma_3 $ is equivalent to adding a new edge to the one of fourteen graphs mentioned in the Lemma \ref{lem3.6}. Hence we only consider the 4-splits of the vertex of degree 5 in $\Gamma_3$. Up to symmetry, we consider the 4-spilt of the vertex $a$. The process of 4-spiltting the vertex $a$ is similar to that of Lemma \ref{lem3.2}.
	
	 {\bf Case 1.}  $B$ $\cap$ \{$b, v_5$\} = $\emptyset$,  \{$b, v_5$\} $\subset$ $A$.
	 
	  Note that $B$=\{$v_1$, $v_6$, $v_7$\}. The set $A$ may be \{$v_5, v_6, b$\}, \{$v_1, v_5, b$\} or \{$v_5, v_7, b$\} in the minimal case. All the resulting graphs contain $K_{3,3}+v$ as a minor (see Figure \ref{fig:18}).
	 	
	 	\begin{figure}[!htb]
	 		\centering
	 		\includegraphics[width=0.7\linewidth]{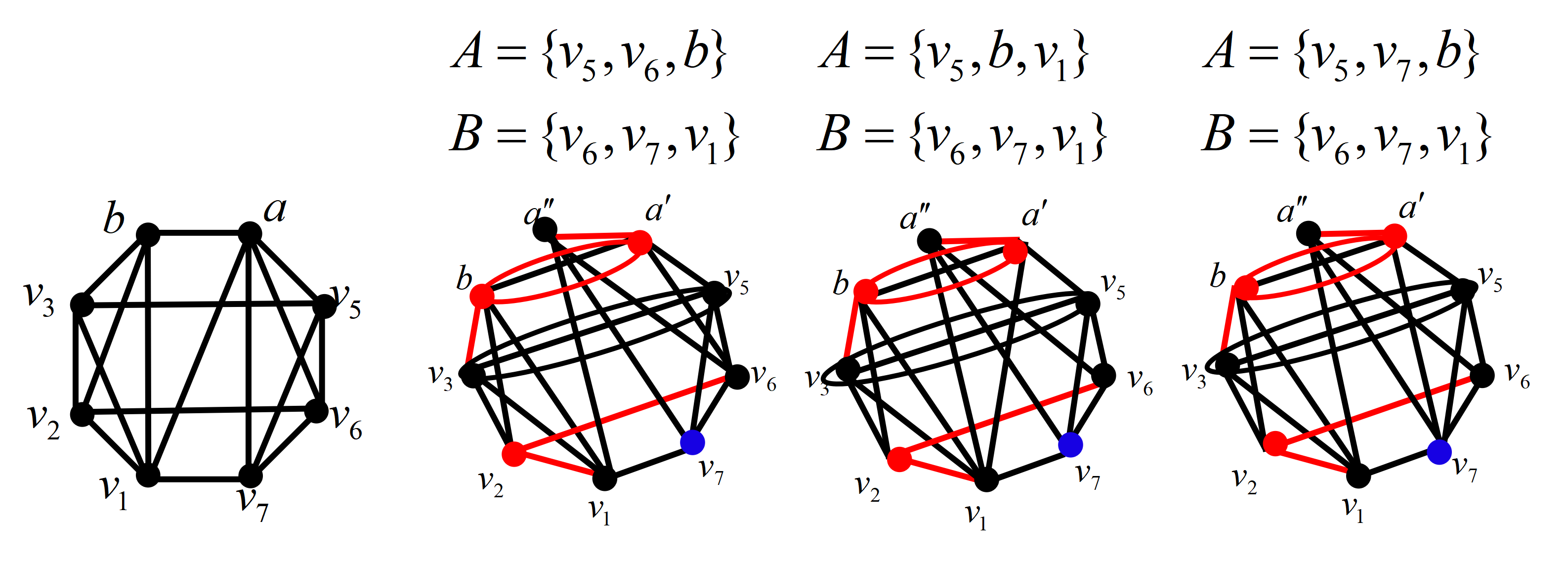}
	 		\caption{ $\Gamma_3$ and three graphs of Case 1}
	 		\label{fig:18}
	 	\end{figure}
 	
   {\bf Case 2.}  $A$ $\cap$ $\{b, v_5\}$ $\neq$ $\emptyset$, $B$ $\cap$ $\{b, v_5\}$ $\neq$ $\emptyset$.
  
  In this case, we assume that $v_5$ $\in$ $A$, $b$ $\in$ $B$ without loss of generality. Since $|N_G(a)|$ = 5, $A$ $\cap$ $B$ $\neq$ $\emptyset$ in the minimal case for $|A|=|B|=3$.

  {\bf Case 2.1.}  $A$ $\cap$ $B$ = \{$v_1$\}
  
 Since $v_1$$\in$$A$ and $v_5$$\in$$A$, one of $v_6$ and $v_7$ must belong to $A$. If $v_6$ $\in$ $A$, then $v_7$ must belong to $B$ in the minimal case. If $v_7$ $\in$ $A$, then $v_6$ must belong to $B$ in the minimal case. We obtain two graphs which contain $K_{3,3}+v$ as a minor (see Figure \ref{fig:19}).
  
  \begin{figure}[!htb]
  	\centering
  	\includegraphics[width=0.4\linewidth]{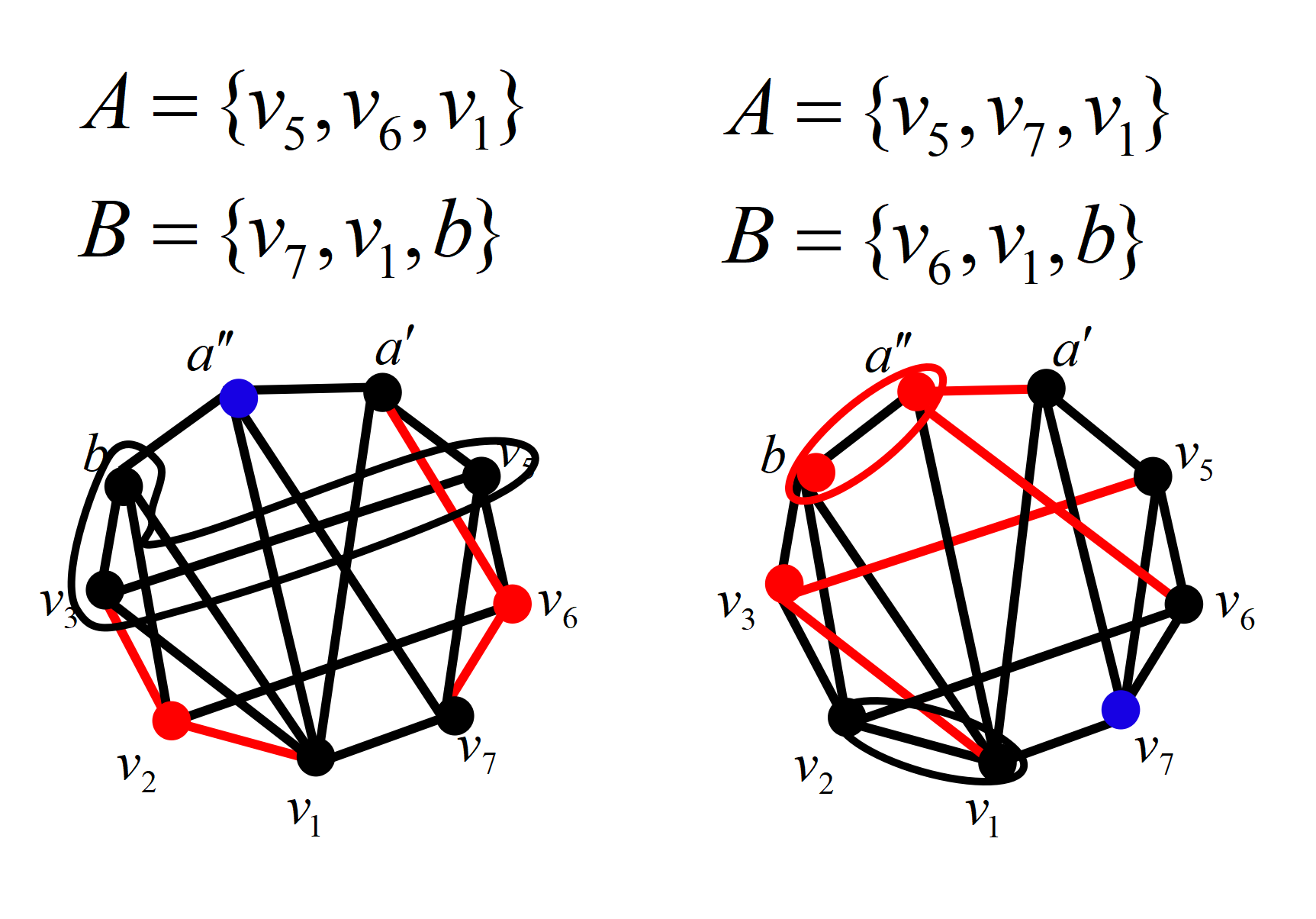}
  	\caption{Two graphs in Case 2.1}
  	\label{fig:19}
  \end{figure}

 {\bf Case 2.2.}  $A$ $\cap$ $B$ = \{$v_6$\}
 
Similarly to the Case 2.1, we have two graphs. It is not difficult to see that both the two resulting graphs contain $ K_{3,3}+v $ as a minor (see Figure \ref{fig:20}).
  
  \begin{figure}[!htb]
  	\centering
  	\includegraphics[width=0.4\linewidth]{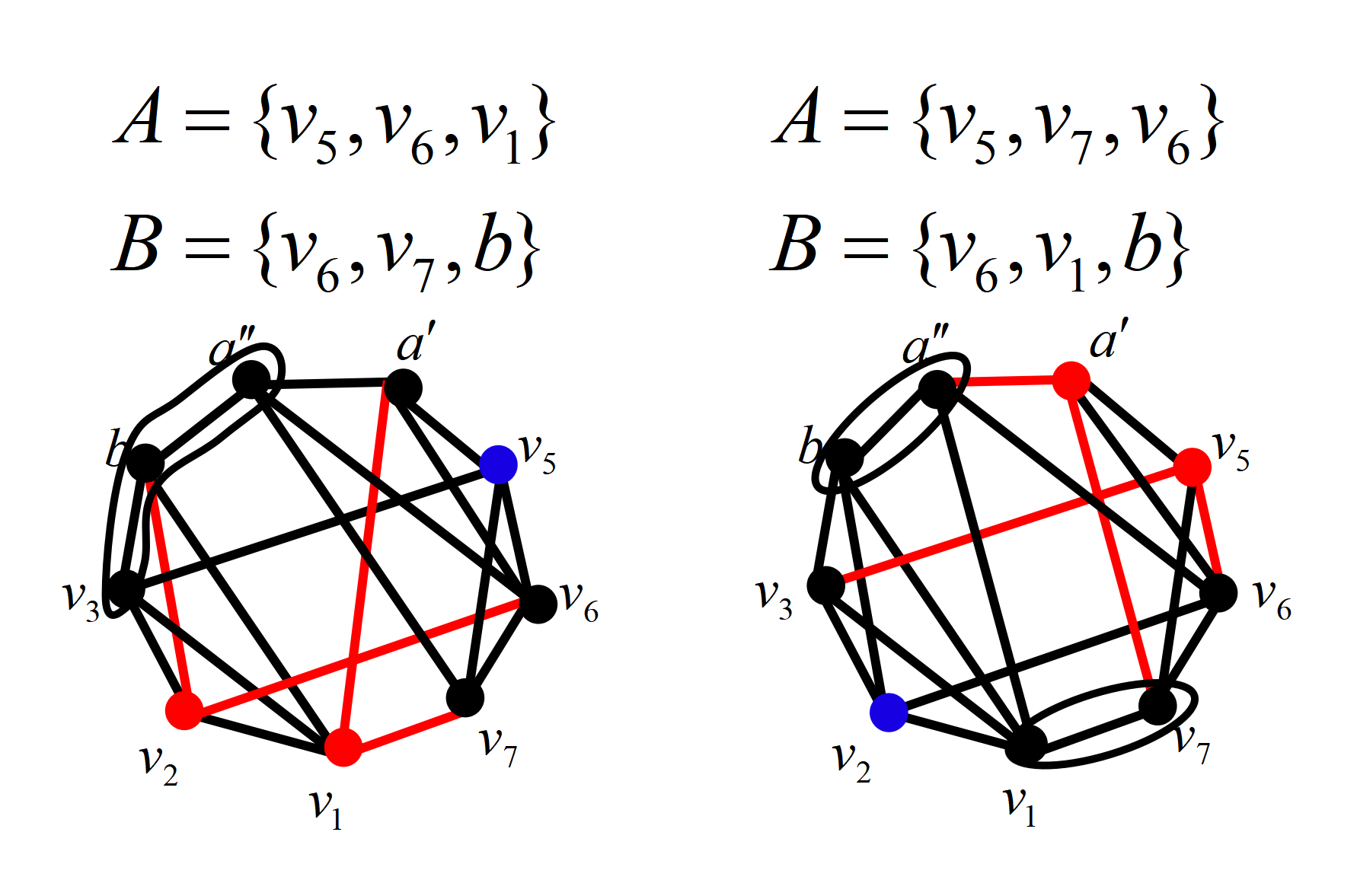}
  	\caption{Two graphs in Case 2.2}
  	\label{fig:20}
  \end{figure}

 {\bf Case 2.3.}  $A$ $\cap$ $B$ = \{$v_7$\}
 
By an argument similar to that of Case 2.1, we get two graphs and verify that both the two resulting graphs contain $ K_{3,3}+v $ as a minor (see Figure \ref{fig:21}).

\begin{figure}[!htb]
	\centering
	\includegraphics[width=0.35\linewidth]{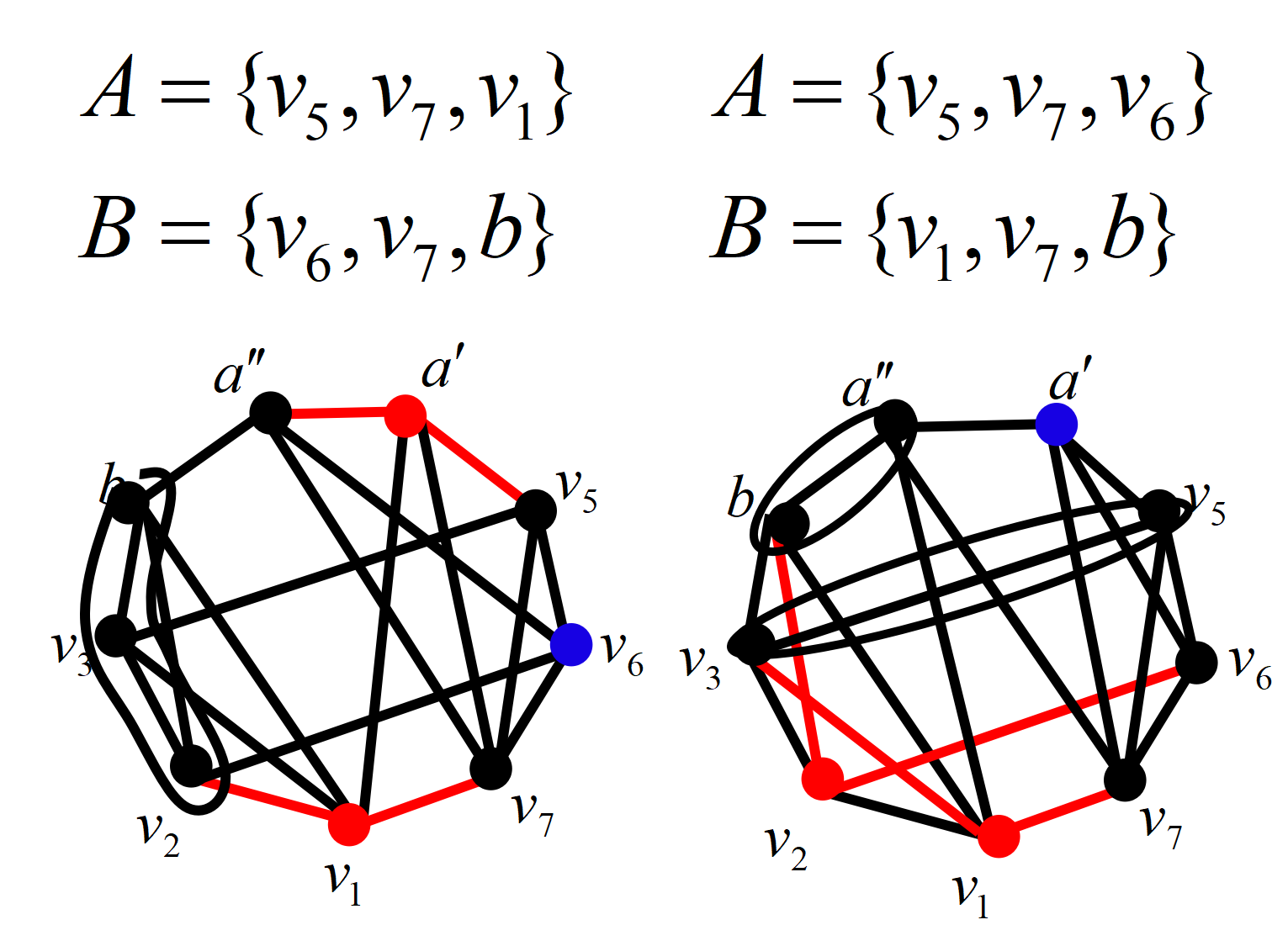}
	\caption{Two graphs in Case 2.3}
	\label{fig:21}
\end{figure}

{\bf Case 2.4.}  $A$ $\cap$ $B$ = \{$v_5$\}

	  	 Since $v_5$$\in$$B$, $b$$\in$$B$, one of $v_1$, $v_6$ and $v_7$ must belong to $B$. First, if $v_1$ $\in$ $B$, then $v_6$ and $v_7$ must belong to $A$ in the minimal case. Second, if $v_6$ $\in$ $B$, then $v_1$ and $v_7$ must belong to $A$ in the minimal case. Third, if $v_7$ $\in$ $B$, then $v_1$ and $v_6$ must belong to $A$ in the minimal case. It is easy to check that all the three resulting graphs contain $ K_{3,3}+v $ as a minor (see Figure \ref{fig:22}).

\begin{figure}[!htb]
	\centering
	\includegraphics[width=0.5\linewidth]{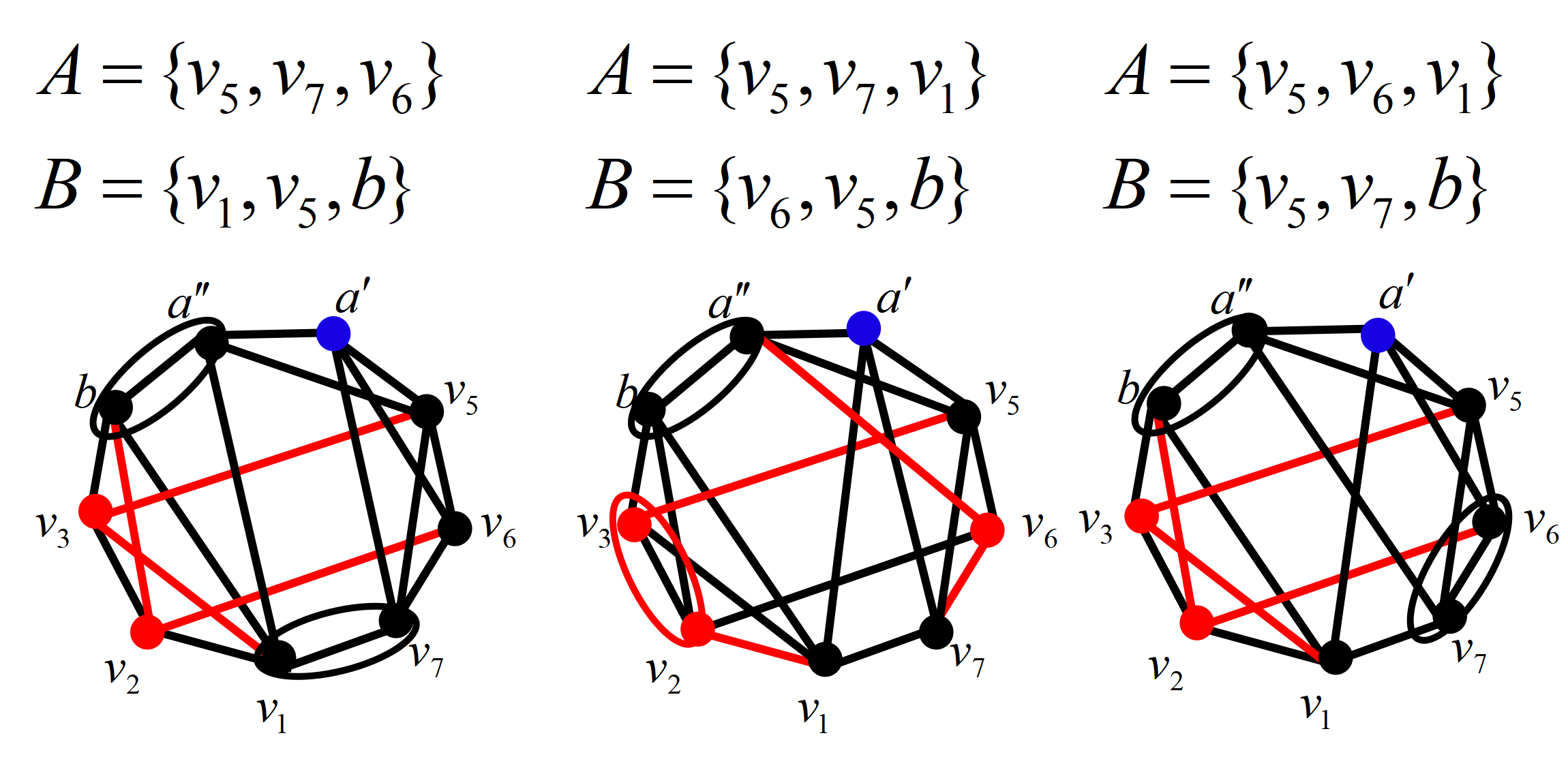}
	\caption{Three graphs in Case 2.4}
	\label{fig:22}
\end{figure}

{\bf Case 2.5.}  $A$ $\cap$ $B$ = \{$b$\}

Similarly to Case 2.4. We obtain three graphs which contain $ K_{3,3}+v $ as a minor as follows (see Figure \ref{fig:23}).

\begin{figure}[!htb]
	\centering
	\includegraphics[width=0.5\linewidth]{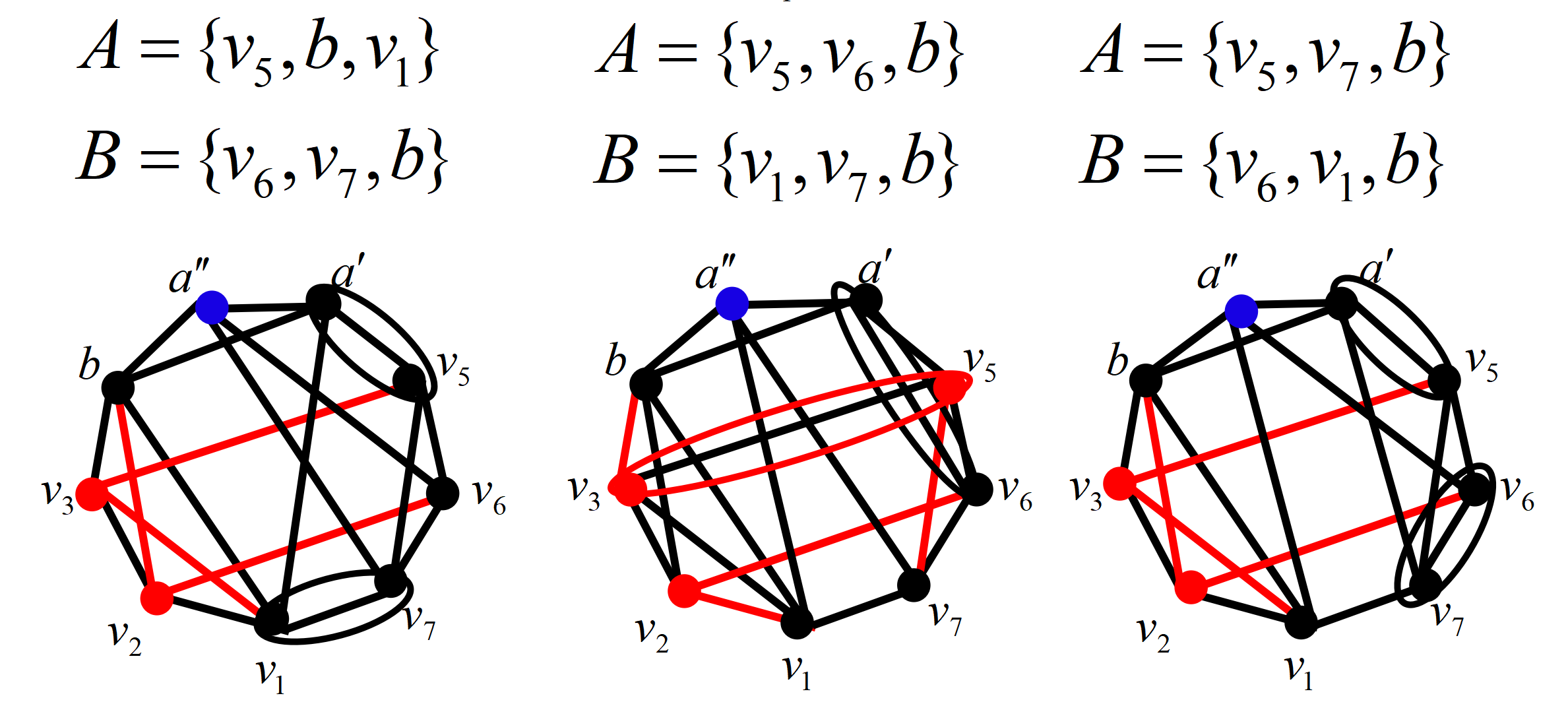}
	\caption{Three graphs in Case 2.5}
	\label{fig:23}
\end{figure}

	In all minimal cases, the new graph $G^{'}$ generated by splitting the vertices of $\Gamma_3$ contains  $ K_{3,3}+v $. Hence all the other graphs obtained by 4-splitting the vertices of $\Gamma_3$ contain $ K_{3,3}+v $ as a minor.

		\end{proof}
	
	 \begin{lem}\label{lem3.8}
	The graph obtained by adding an edge to $DW_5$ contains $ K_{3,3}+v $.
		
	\end{lem}

\begin{proof}
	By symmetry, we only need to consider adding $av_5$, $ bv_2$ or $ v_2v_5$. We denote the resulting graphs by $DW_5^1$, $DW_5^2$ and $DW_5^3$ as shown in Figure \ref{fig:24}. It is straightforward to verify that $DW_5^1$, $DW_5^2$ and $DW_5^3$ all contain $K_{3,3}+v$ as a minor.
	
\begin{figure}[!htb]
	\centering
	\includegraphics[width=0.7\linewidth]{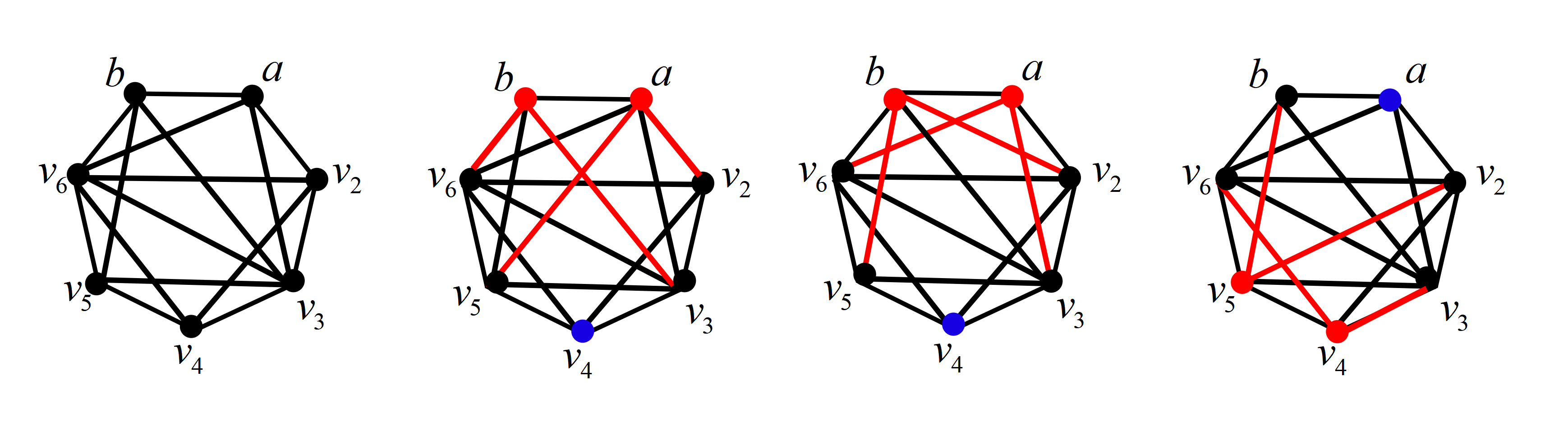}
	\caption{$DW_5$, $DW_5^1$, $DW_5^2$, $DW_5^3$}
	\label{fig:24}
\end{figure}

	\end{proof}

	\begin{lem}\label{lem3.16}
	The only $ K_{3,3}+v $-$ minor $-$ free $ graph is  $DW_5$ by 4-splitting a vertex of degree 4 in $DW_4$. 
\end{lem}

\begin{proof}
		Let \{$v_1,v_2,\ldots,v_6$\} be vertices of $DW_4$ (shown in Figure \ref{fig:02}). By symmetry, we consider the 4-split of the vertex $v_1$. Note that $d(v_1)=4$. Similarly to the Lemma \ref{lem3.3}, we obtain seven graphs which contain $K_{3,3}+v$ as a minor with the exception of  $DW_4$ (see in Figure \ref{fig:34}).
		
		\begin{figure}[!htb]
			\centering
			\includegraphics[width=0.7\linewidth]{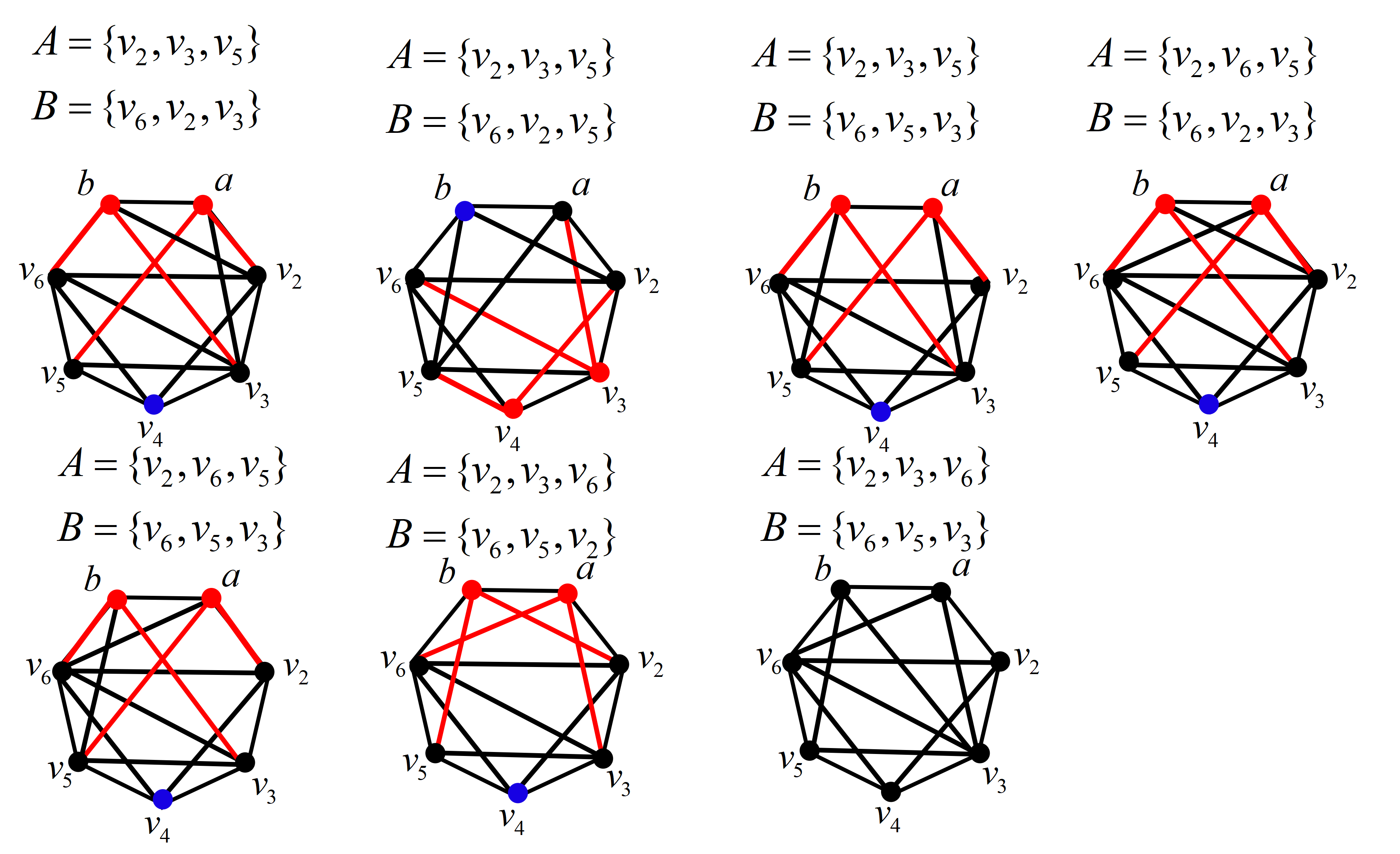}
			\caption{Seven graphs obtained by splitting the vertex $v_1$ in $DW_4$. }
			\label{fig:34}
		\end{figure}
	
The graph obtained in non-minimal cases is equivalent to add edges $av_i$ or $bv_j$ to the above graphs mentioned in Figure \ref{fig:34}. Hence by Lemma \ref{lem3.8}, all the other new graphs obtained by 4-splitting the vertex $v_1$ in $DW_4$ contain $ K_{3,3}+v $ as a minor.	
	
	\end{proof}

 \begin{lem}\label{lem3.9}
	The graph obtained by adding an edge to $DW_n$ $($$n$ $\ge$ 6$)$ contains $ K_{3,3}+v $.
	
\end{lem}

\begin{proof}
	 Let \{$v_1,v_2,\ldots,v_n$\} be vertices of $DW_n$ (shown in Figure \ref{fig:25}). There are two cases for the edge $v_iv_j$.
	 
	    {\bf Case 1.}  $v_i$ = $v_1$.
	
		{\bf Case 1.1.}  $j$ = 3.
		
		The new graph $G^{'}$ generated by adding $v_1v_3$ to $DW_n$ contains $DW_5^1$ as a minor (see $DW_n^1$ in Figure \ref{fig:25}).
		
		{\bf Case 1.2.}  4 $\leq$ $j$ $\leq$ $n - 3$ .
		
		The new graph $G^{'}$ generated by adding $v_1v_j$ to $DW_n$ contains $DW_5^3$ as a minor (see $DW_n^2$ in Figure \ref{fig:25}).
			
		By symmetry, the case of $v_iv_{n-2}$ (2 $\leq$ $i$ $\leq$ $n-4$) is similar to the above.
			
		{\bf Case 2.}  $v_i$ $\neq$ $v_1$ and $v_j$ $\neq$ $v_{n-2}$.
			
		The new graph $G^{'}$ generated by adding $v_iv_j$ to $DW_n$ contains $DW_5^2$ as a minor (see $DW_n^3$ in Figure \ref{fig:25}).
			
		Thus, the graph obtained by adding an edge to $DW_n$ ($n$ $\ge$ 6) contains $ K_{3,3}+v $ as a minor.	
		
\begin{figure}[!htb]
	\centering
	\includegraphics[width=0.7\linewidth]{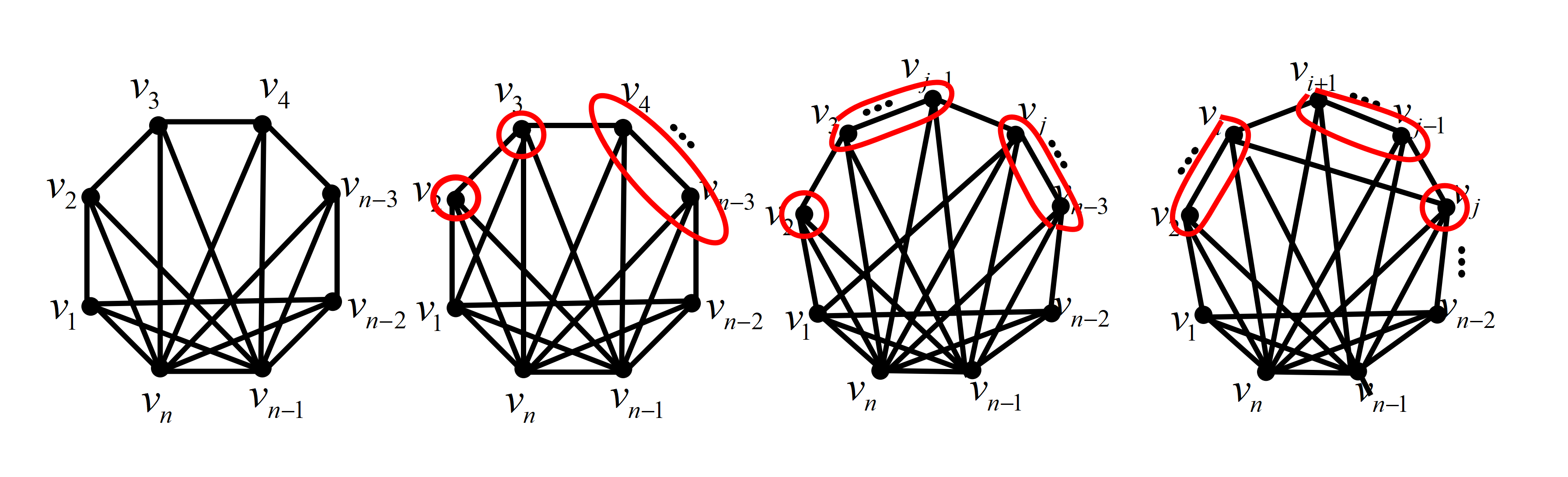}
	\caption{$DW_n$, $DW_n^1$, $DW_n^2$, $DW_n^3$}
	\label{fig:25}
\end{figure}

\end{proof}
\begin{lem}\label{lem3.10}
   Every 4-split of a vertex of degree 6 in $DW_5$ contains a $ K_{3,3}+v $-$ minor$.
	
\end{lem}

\begin{proof}

	Let \{$v_1,v_2,\ldots,v_7$\} be vertices of $DW_5$ (shown in Figure \ref{fig:26}). By symmetry, we split the vertex $v_6$. We first consider the minimal case for $A$ and $B$, where $A$, $B$ are the subsets of the $N_G(v_6)$, $A \cup B$=$N_G(v_6)$. Suppose that $a, b$ are the new vertices obtained by 4-splitting the vertex $v_6$. Since $|N_G(v_6)|$ = 6, the process of 4-splitting the vertex $v_6$ is similar to that of Lemma \ref{lem3.5}. Hence we have three cases for $A$ and $B$. The first case is that one of $A$ $\cap$ \{$v_5, v_7$\} and $B$ $\cap$ \{$v_5, v_7$\} is $\emptyset$. Without loss of generality, we assume that $B$ $\cap$ \{$v_5, v_7$\} is $\emptyset$. The second case is that $A$ $\cap$ \{$v_5, v_7$\} $\neq$ $\emptyset$, $B$ $\cap$ \{$v_5, v_7$\} $\neq$ $\emptyset$ and $A$ $\cap$ $B$ $=$ $\emptyset$. The third case is that $A$ $\cap$ \{$v_5, v_7$\} $\neq$ $\emptyset$, $B$ $\cap$ \{$v_5, v_7$\} $\neq$ $\emptyset$ and $A$ $\cap$ $B$ $\neq$ $\emptyset$.
	
	{\bf Case 1.}  $B$ $\cap$ \{$v_5, v_7$\} = $\emptyset$.
	 
	Similarly to the Case 1 in Lemma \ref{lem3.5}, we have four graphs. It is not difficult to see that all the four resulting graphs contain $ K_{3,3}+v $ as a minor (see Figure \ref{fig:26}).
		
		\begin{figure}[!htb]
			\centering
			\includegraphics[width=0.7\linewidth]{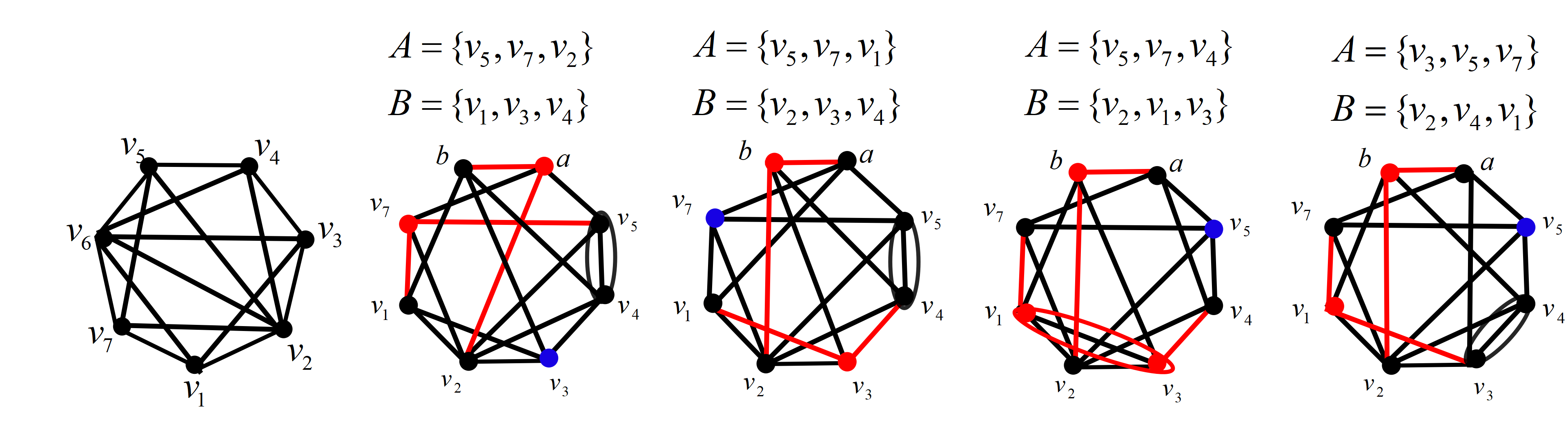}
			\caption{$DW_5$ and four graphs in Case 1}
			\label{fig:26}
		\end{figure}
	
	{\bf Case 2.}  $A$ $\cap$ \{$v_5, v_7$\} $\neq$ $\emptyset$, $B$ $\cap$ \{$v_5, v_7$\} $\neq$ $\emptyset$ and $A$ $\cap$ $B$ $=$ $\emptyset$
	
	We consider the minimal case for $|A|=|B|=3$ and assume that $v_5$$\in$$A$, $v_7$$\in$$B$ without loss of generality. By an argument argument similar to that of Case 2 in Lemma \ref{lem3.5}, we have six graphs. It is easy to check that all the resulting graphs contain $ K_{3,3}+v $ as a minor (see Figure \ref{fig:27}).
	
\begin{figure}[!htb]
	\centering
	\includegraphics[width=0.5\linewidth]{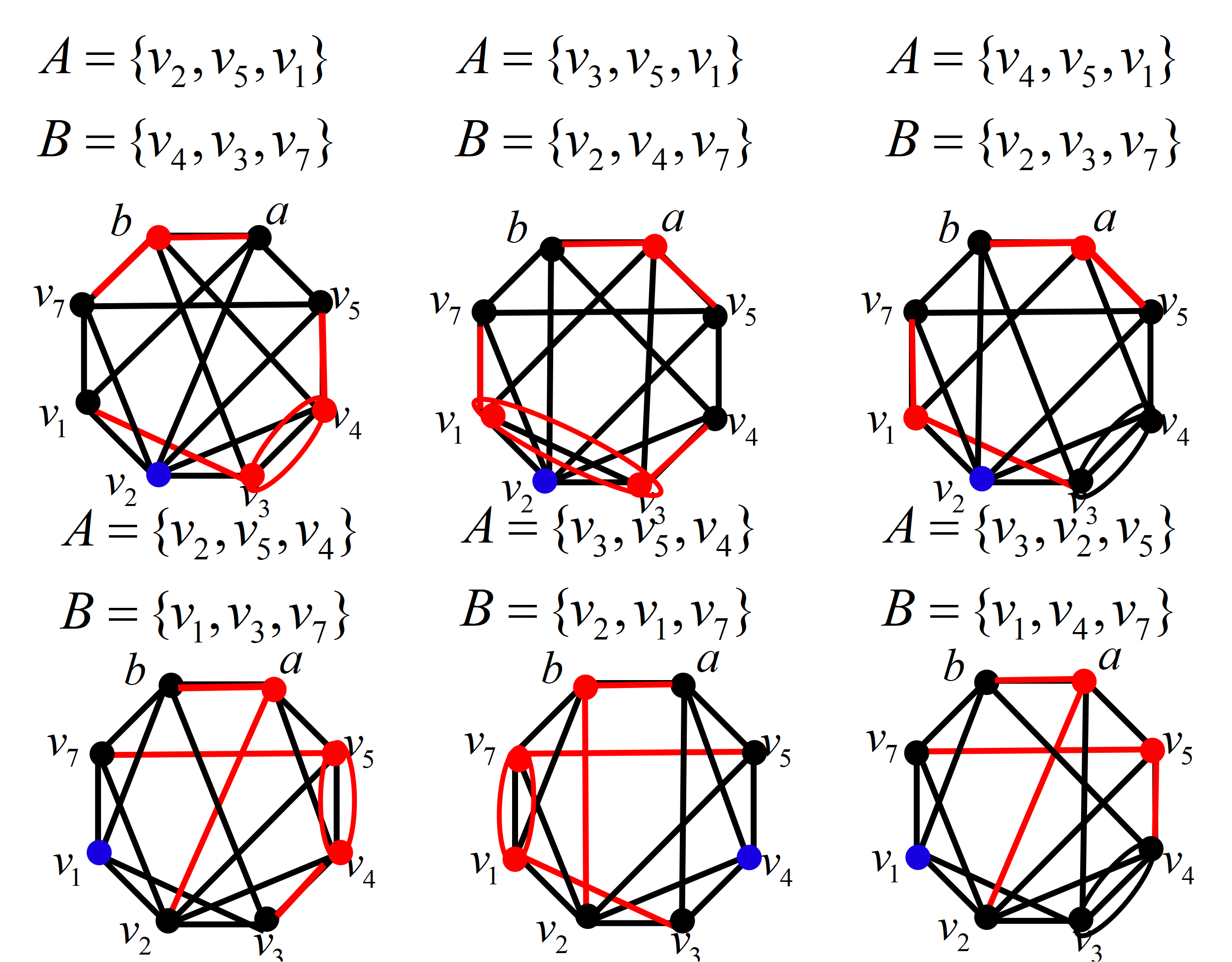}
	\caption{Six graphs of Case 2}
	\label{fig:27}
\end{figure}

 {\bf Case 3.}  $A$ $\cap$ \{$v_5, v_7$\} $\neq$ $\emptyset$, $B$ $\cap$ \{$v_5, v_7$\} $\neq$ $\emptyset$ and $A$ $\cap$ $B$ $\neq$ $\emptyset$.
 
 	We consider the minimal case for $A$ and $B$. We assume that $v_5$$\in$$A$, $v_7$$\in$$B$ without loss of generality.

{\bf Case 3.1.}  $A$ $\cap$ $B$ $=$ $\{v_1\}$

Since $|N_G(v_6)|$=6 and $A$ $\cap$ $B$ $=$ $\{v_1\}$, the set $A$ or set $B$ contains 4 vertices in the minimal case. The graph generated in this case is equivalent to add the edge $av_1$ or $bv_1$ to the above graphs mentioned in Case 2. Thus, the graphs all contain $ K_{3,3}+v $ as a minor. The cases of $A$ $\cap$ $B$ $=$ $\{v_2\}$, $A$ $\cap$ $B$ $=$ $\{v_3\}$ and $A$ $\cap$ $B$ $=$ $\{v_4\}$ are similar to the above.

{\bf Case 3.2.}  $A$ $\cap$ $B$ $=$ $\{v_5\}$

Note that $d(v_6)=6$ and $A$ $\cap$ $B$ $=$ $\{v_5\}$. One of the set $A$ and set $B$ contains 4 vertices in the minimal case. The graph generated in this case is equivalent to add the edge $av_5$ or $bv_5$ to the above graphs mentioned in Case 1 or Case 2. Thus, the graphs all contain $K_{3,3}+v$ as a minor. The case of $A$ $\cap$ $B$ $=$ $\{v_7\}$ is similar to the above.

In all minimal cases, the new graph $G^{'}$ generated by 4-splitting a vertex of degree 6 in $DW_5$ contains $ K_{3,3}+v $. Hence all the other graphs obtained by 4-spliting a vertex of degree 6 in $DW_5$ contain $ K_{3,3}+v $ as a minor.

\end{proof}

Let $\cal F$ =\{ $F$ : $F$ is obtained by 4-spiliting a vertex of degree 6 in $DW_5$ \}.

\begin{lem}\label{lem3.11}
	The graph obtained by splitting the vertex of degree $n-1$ in $DW_n$ $($$n$ $\geq$ 6$)$ contains a $ K_{3,3}+v $-$ minor$. 
	
\end{lem}

\begin{proof}
Let \{$v_1,v_2,\ldots,v_n$\} be vertices of $DW_n$ (shown in Figure\ref{fig:30}).	By symmetry, we consider the 4-splitting of the vertex $v_n$. By contrasting the structure of $DW_5$$\backslash$$v_7$ and $DW_n$$\backslash$$v_n$ (see Figure \ref{fig:30}), we conclude that the new graph $G^{'}$ generated by 4-splitting the vertex of degree $n-1$ in $DW_n$ contains a graph $F$ as a minor. Thus, $G^{'}$ must contain $K_{3,3}+v $ as a minor.

\begin{figure}[!htb]
	\centering
	\includegraphics[width=0.7\linewidth]{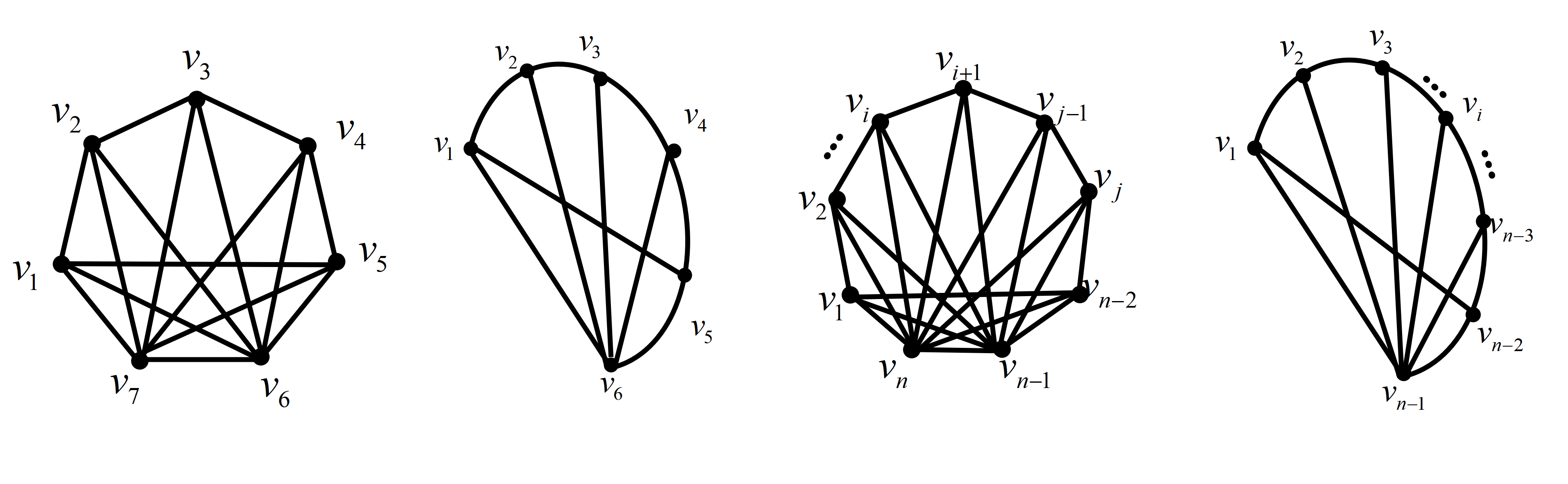}
	\caption{$DW_5$, $DW_5$$\backslash$$v_7$, $DW_n$, $DW_n$$\backslash$$v_n$}
	\label{fig:30}
\end{figure}

\end{proof}

	\begin{lem}\label{lem3.12}
	The graph obtained by 4-splitting the vertex of degree 4 in $DW_5$ contains a $ K_{3,3}+v $-$ minor$ with the exception of $DW_6$. 
\end{lem}

\begin{proof}
	We consider the minimal case for $|A|=|B|=3$, where $A, B$ are the subsets of the $N_G(v)$, $A \cup B$=$N_G(v)$. Let \{$v_1,v_2,\ldots,v_7$\} be vertices of $DW_5$ (shown in Figure \ref{fig:31}). By symmetry, we split the vertices of $v_4$, $v_3$ and $v_1$ respectively. Note that $d(v_1)=4$, $d(v_3)=4$ and $d(v_4)=4$. Thus the process of 4-splitting the vertex $v_1$, $v_3$ or $v_4$ is similar to that of Lemma \ref{lem3.3}.

	{\bf Case 1.}  Split the vertex $v_4$.
	
		$A$ $\cap$ $\{v_3, v_5\}$ $\neq$ $\emptyset$ and $B$ $\cap$ $\{v_3, v_5\}$ $\neq$ $\emptyset$ in the minimal case for $|A|=|B|=3$. Without loss of generality, we assume that $v_5$ $\in$ $A$, $v_3$ $\in$ $B$. We get seven graphs which contain $ K_{3,3}+v $ as a minor (see Figure \ref{fig:31}).
		
\begin{figure}[!htb]
	\centering
	\includegraphics[width=0.7\linewidth]{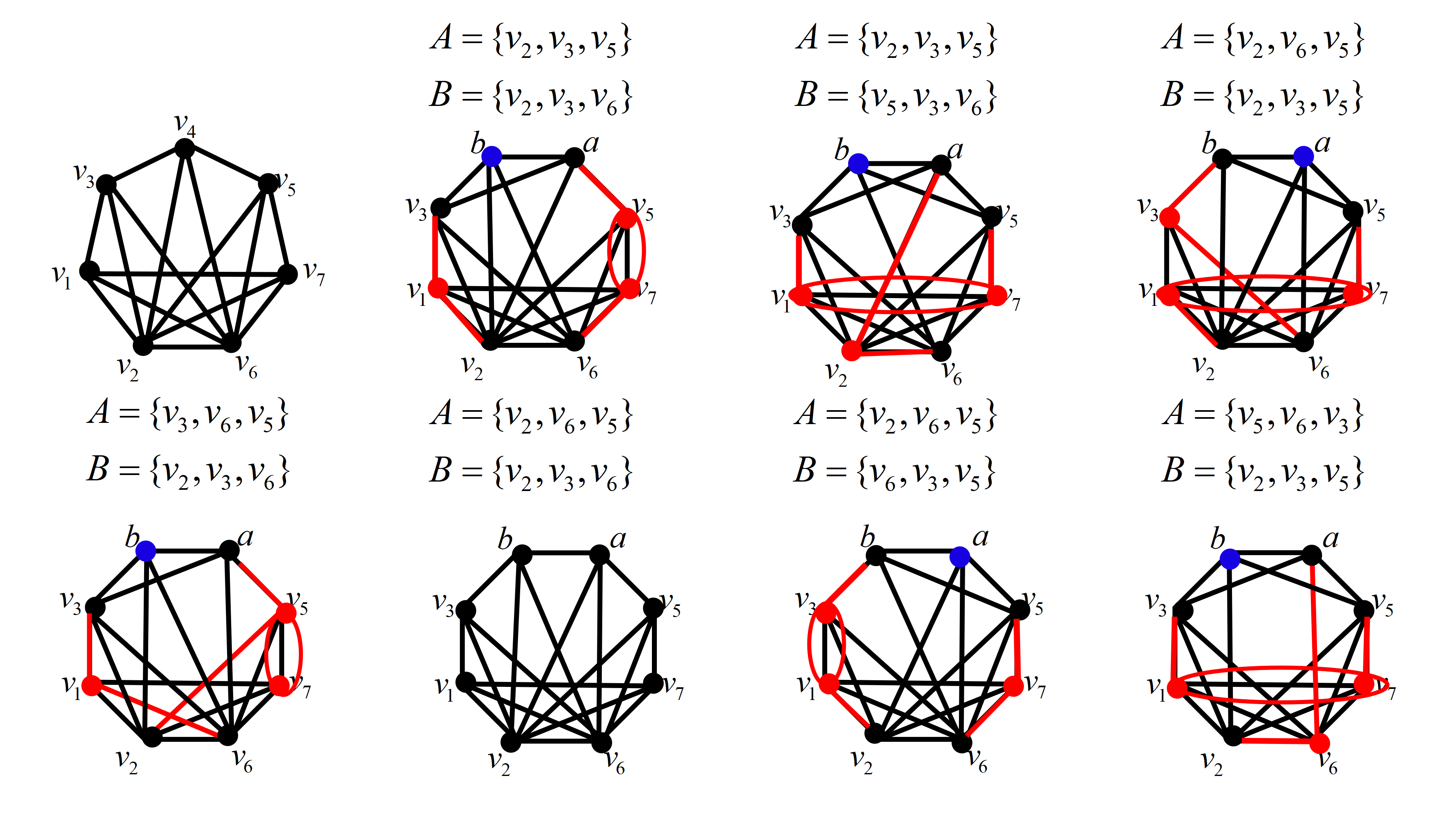}
	\caption{$DW_5$ and seven graphs in Case 1}
	\label{fig:31}
\end{figure}

		{\bf Case 2.}  Split the vertex $v_3$.
		
	$A$ $\cap$ $\{v_1, v_4\}$ $\neq$ $\emptyset$ and $B$ $\cap$ $\{v_1, v_4\}$ $\neq$ $\emptyset$ in the minimal case for $|A|=|B|=3$. Without loss of generality, we assume that $v_4$ $\in$ $A$, $v_1$ $\in$ $B$ by symmetry. It is clear that all the seven resulting graphs contain $ K_{3,3}+v $ as a minor (see Figure \ref{fig:32}).
		
\begin{figure}[!htb]
	\centering
	\includegraphics[width=0.7\linewidth]{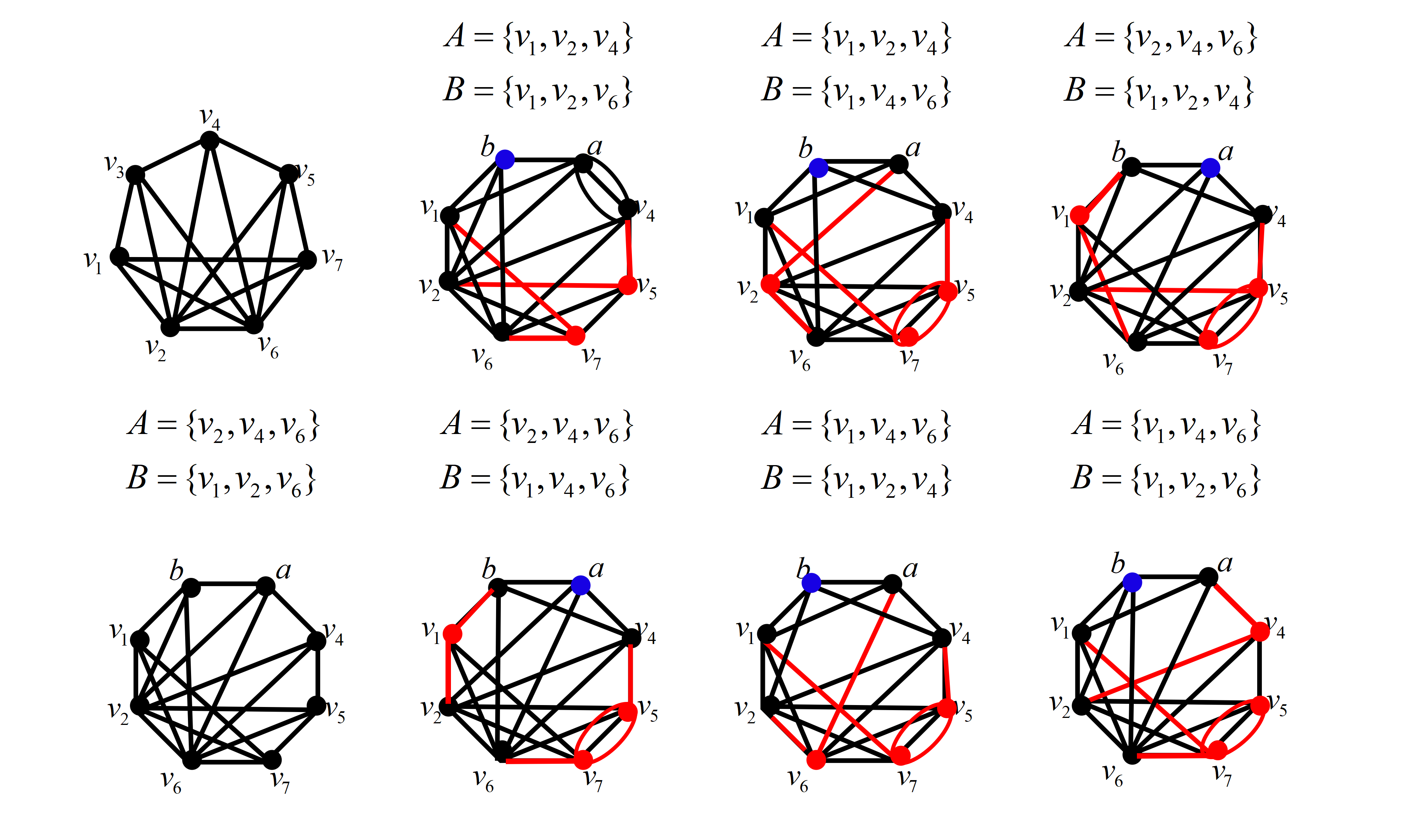}
	\caption{$DW_5$ and seven graphs in Case 2}
	\label{fig:32}
\end{figure}
	
	{\bf Case 3.}  Split the vertex $v_1$.
	
	$A$ $\cap$ $\{v_2, v_3\}$ $\neq$ $\emptyset$ and $B$ $\cap$ $\{v_2, v_3\}$ $\neq$ $\emptyset$ in the minimal case for $|A|=|B|=3$. It is not difficult to see that all the seven resulting graphs contain $ K_{3,3}+v $ as a minor (see Figure \ref{fig:33}).
	
\begin{figure}[!htb]
	\centering
	\includegraphics[width=0.7\linewidth]{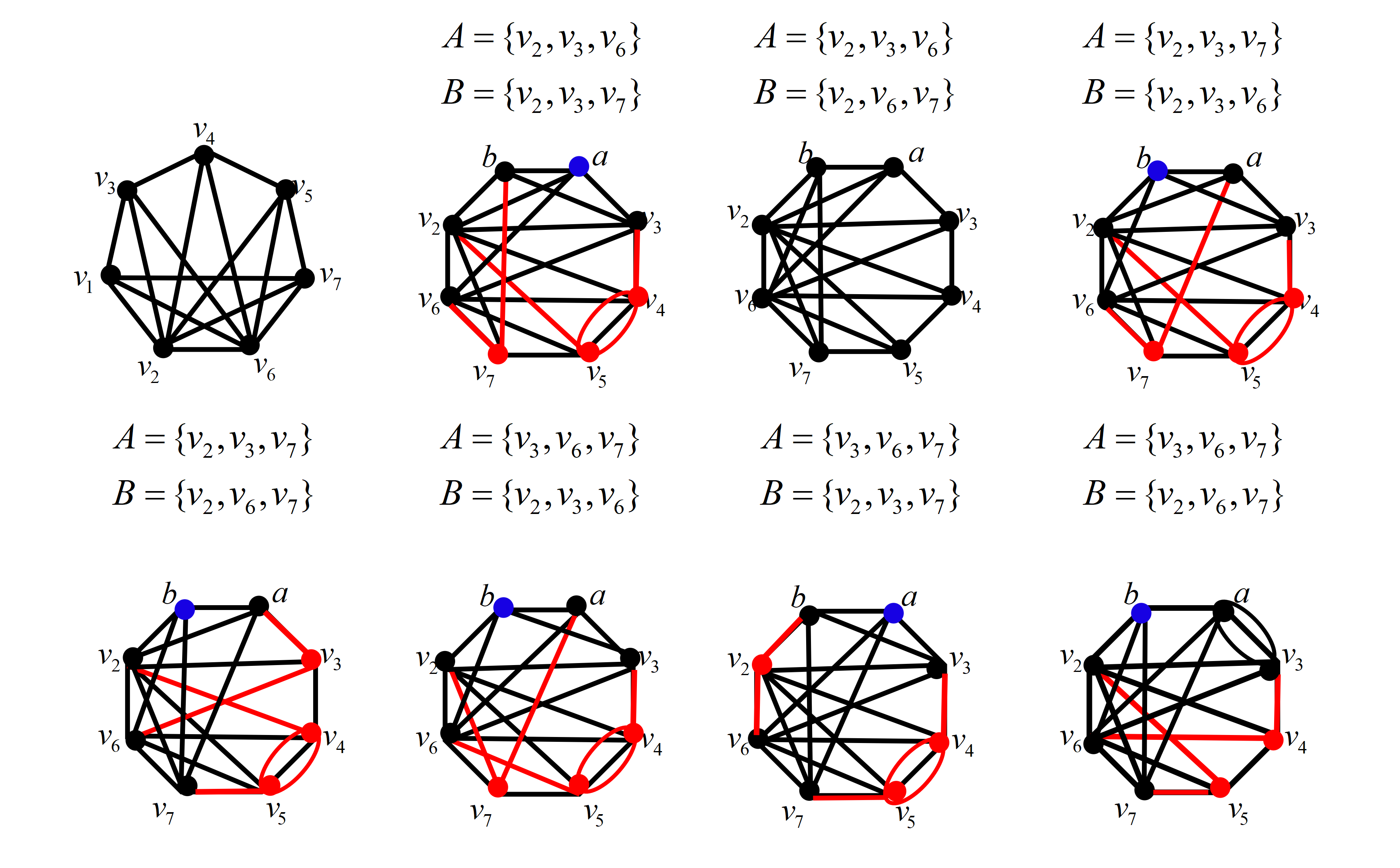}
	\caption{$DW_5$ and seven graphs in Case 3}
	\label{fig:33}
\end{figure}

   	In all minimal cases, the new graph $G^{'}$ generated by 4-splitting the vertices of degree 4 in $DW_5$ with $|A|=|B|=3$ besides $DW_6$ contains $ K_{3,3}+v $. In the other non-minimal cases, it is equivalent to add new edges $av_i$ or $bv_j$ to the above graphs mentioned in Figure \ref{fig:31}--\ref{fig:33}. According to Lemma \ref{lem3.9} and above results, the graph obtained by 4-splitting the vertex of degree 4 in $DW_5$ contains $ K_{3,3}+v $ as a minor besides $DW_6$.

\end{proof}

\begin{lem}\label{lem3.13}
	The graph obtained by 4-splitting the vertex of degree 4 in $DW_n$ contains a $K_{3,3}+v$-$ minor$ $($$n$$\geq$$6$$)$ besides $DW_{n+1}$. 
\end{lem}

\begin{proof}
	 It is straightforward to verify that 4-splitting a vertex of degree 4 in $DW_n$ can be considered as 4-spliting a vertex of degree 4 in $DW_5$. Hence the new graph $G^{'}$ generated by 4-splitting the vertice of degree 4 in $DW_n$ in the minimal case contains $K_{3,3}+v$ as a minor with the exception of $DW_{n+1}$. Since the graph obtained in the non-minimal cases is equivalent to adding edges to the above graphs, it contains $K_{3,3}+v$ as a minor. Thus every graph obtained by 4-splitting the vertex of degree 4 in $DW_n$ contains $ K_{3,3}+v $ as a minor besides $DW_{n+1}$. 
\end{proof}

\begin{lem}\label{lem3.14}
	Every 4-split of $K_6$$\backslash$$e$ contains a $ K_{3,3}+v $-$ minor$.
	
\end{lem}
\begin{proof}
	 According to the Lemma \ref{lem3.1} and Lemma \ref{lem3.8}, the graph obtained by adding an edge to $\Gamma_1$ or $M_7$ contains $K_{3,3}+v $ as a minor. $K_6$$\backslash$$e$ is a graph by adding an edge to the $DW_4$. The new graph $G^{'}$ generated by 4-splitting vertices of $K_6$$\backslash$$e$ can be regarded as a graph by adding an edge to the graph  generated by 4-splitting vertices of $DW_4$. Thus, every 4-split of $K_6$$\backslash$$e$ contains $K_{3,3}+v$ as a minor.
\end{proof}

\begin{lem}\label{lem3.15}
	Every 4-split of $K_6$ contains a $ K_{3,3}+v $-$ minor$.
	
\end{lem}
\begin{proof}
	According to the Lemma \ref{lem3.1} and Lemma \ref{lem3.8}, the graph obtained by adding an edge to $\Gamma_1$ or $M_7$ contains $K_{3,3}+v $ as a minor. $K_6$ is a graph by adding two edges to the $DW_4$. The new graph $G^{'}$ generated by 4-splitting vertices of $K_6$ can be regarded as a graph by adding two edges to the graph generated by 4-splitting vertices of $DW_4$. Thus, every 4-split of $K_6$ contains $K_{3,3}+v$ as a minor.
\end{proof}

\textbf{Proof of Theorem \ref{thm1.2}}. The result follows from Lemmas \ref{lem3.1}---\ref{lem3.15}. \qed

\end{document}